\documentclass[bj]{imsart}

\RequirePackage{amsthm,amsmath,amsfonts,amssymb}
\RequirePackage[numbers]{natbib}
\RequirePackage[colorlinks,citecolor=blue,urlcolor=blue]{hyperref}
\RequirePackage{graphicx}
\RequirePackage{bbm,dsfont}
\RequirePackage{comment}

\startlocaldefs

 \newtheorem{thm}{Theorem}[section]
 \newtheorem{cor}[thm]{Corollary}
 \newtheorem{lem}[thm]{Lemma}
 \newtheorem{prop}[thm]{Proposition}

 \theoremstyle{definition}
 
 \newtheorem{rem}[thm]{Remark}
 \newtheorem{ex}[thm]{Example}
 \newtheorem{ass}[thm]{Assumption}

\newtheorem{thmx}{Question}[section]

\renewcommand{\ss}{s}

 \newcommand{\lambdav}{\mbox{\boldmath$\lambda$}}

 \newcommand{\ind}{\mathbbm{1}}
 
 \newcommand{\Pro}{\mathcal{P}}

 \newcommand{\R}{\mathds{R}}
 \newcommand{\Z}{\mathds{Z}}

 \newcommand{\ic}{\mathrm{i}}
 
 \newcommand{\FF}{\mathbb{F}}

 \newcommand{\E}{\mathbb{E}}

 \newcommand{\N}{\mathds{N}}

 \newcommand{\OO}{{O}}

 \newcommand{\oo}{\mbox{\scriptsize ${O}$}}

 \newcommand{\sd}{\mathfrak{s}}
 \newcommand{\ad}{\mathfrak{a}}
 \newcommand{\bd}{\mathfrak{b}}
 \newcommand{\cd}{\mathfrak{c}}

  \renewcommand{\ss}{s}

\renewcommand{\P}{\mathds{P}}
\newcommand{\Aone}{{\bf (B1)}}
\newcommand{\Atwo}{{\bf (B2)}}
\newcommand{\Athree}{{\bf (B3)}}

\newcommand{\Bone}{{\bf (A1)}}
\newcommand{\Btwo}{{\bf (A2)}}
\newcommand{\Bthree}{{\bf (A3)}}

\endlocaldefs

\begin{document}

\begin{frontmatter}
\title{A Berry-Esseen bound with (almost) sharp dependence conditions}
\runtitle{A Berry-Esseen bound with (almost) sharp dependence conditions}

\begin{aug}
\author[A]{\fnms{Moritz} \snm{Jirak}\ead[label=e1]{moritz.jirak@univie.ac.at}},
\address[A]{University of Vienna.
\printead{e1}}

\end{aug}

\begin{abstract}
Suppose that the (normalised) partial sum of a stationary sequence converges to a standard normal random variable.
Given sufficiently moments, when do we have a rate of convergence of $n^{-1/2}$ in the uniform metric, in other words, when do we have the optimal Berry-Esseen bound? We study this question in a quite general framework and find the (almost) sharp dependence conditions. The result applies to many different processes and dynamical systems. As specific, prominent examples, we study functions of the doubling map 2x mod 1, the left random walk on the general linear group and functions of linear processes.
\end{abstract}


\end{frontmatter}

\section{Introduction}\label{sec_intro}

Suppose that a stationary sequence $(X_k)_{k \in \Z}$ satisfies the CLT with variance $\ss^2 > 0$, that is,
\begin{align}\label{CLT}
\frac{1}{\sqrt{n}} S_n \xrightarrow{w} \mathcal{N}\big(0,\ss^2\big), \quad S_n \stackrel{def}{=} \sum_{k = 1}^n X_k.
\end{align}

The question of the rate of convergence in \eqref{CLT} has been extensively studied in the literature under numerous different setups with respect to some metric for probability measures and underlying structure of the sequence $(X_k)_{k \in \Z}$. Perhaps one of the more important metrics is the uniform (Kolmogorov) metric, given as
\begin{align}\label{eq_clt_rate_ks}
\Delta_n \stackrel{def}{=} \sup_{x \in \R} \bigl|\P\bigl(S_n \leq x \sqrt{n \ss^2} \bigr) - \Phi(x)\bigr|.
\end{align}
Subject to various notions of weak-dependence and additional regularity conditions, the optimal rate of convergence $n^{-1/2}$ has been reached in \cite{bolthausen_1982_ptrf}, \cite{Rio_1996}, \cite{nagaev:1957}, \cite{GOUEZEL2005:hip}, \cite{herve_pene_2010_bulletin}, \cite{jirak_be_aop_2016}, \cite{meyn_kontoyiannis_2003aap}, to name a few, but this list is far from being complete. In view of such results, the following question naturally arises.

\begin{thmx}\label{QA}
What are the sharp (weak) dependence conditions, such that the optimal rate $n^{-1/2}$ is reached, that is,
\begin{align}\label{optimal:question}
\Delta_n \leq C n^{-1/2}
\end{align}
for some constant $C > 0$?
\end{thmx}
In this note, we (essentially) solve this question in a general framework, improving upon many results in the literature.
We achieve this by refining the approach of \cite{jirak_be_aop_2016}. However, the solution shows that the question of sharp dependence conditions is, in general, much more complicated, leading to a multitude of open problems.

\section{Main results}\label{sec_main}

Throughout this note, we use the following notation. For a random variable $X$ and $p \geq 1$, we denote with $\|X\|_p = (\E|X|^p)^{1/p}$ the ${\mathds{L}}^p$ norm. We use $\lesssim$, $\gtrsim$, ($\thicksim$) to denote (two-sided) inequalities involving a multiplicative constant, which may depend (simultaneously) on any of the quantities appearing in Assumption \ref{ass_main_dependence} below. This also applies whenever we use the Landau symbols $o$ and $O$. For $a,b \in \R$, we put $a \wedge b = \min\{a,b\}$, $a \vee b = \max\{a,b\}$. For random variables $X,Y$ we write $X \stackrel{d}{=} Y$ for equality in distribution.
\\
\\
Consider a sequence of real-valued random variables $X_1,\ldots,X_n$. It is well known (cf. \cite{rosenblatt:book}), that, on a possibly larger probability space, this sequence can be represented as
\begin{align}\label{eq_structure_condition}
{X}_{k} = g_k\bigl(\epsilon_{k}, \epsilon_{k-1}, \ldots \bigr), \quad 1 \leq k \leq n,
\end{align}
for some measurable functions $g_k$\footnote{In fact, $g_k$ can be selected as a map from $\R^k$ to $\R$}, where $(\epsilon_k)_{k \in \Z}$ is a sequence of independent and identically distributed random variables. At this stage, it is also worth mentioning that any real valued K-automorphism can be represented as in \eqref{eq_structure_condition}, a consequence of Vershik's famous \textit{theorem on lacunary isomorphism}, see for instance \cite{emery_schachermayer_2001}, \cite{vershik_doc_transl}. This motivates the following setup: Denote the corresponding $\sigma$-algebra with $\mathcal{E}_k = \sigma( \epsilon_j, \, j \leq k)$. Given a real-valued stationary sequence $({X}_k)_{k\in \Z}$, we always assume that $X_k$ is adapted to $\mathcal{E}_{k}$ for each $k \in \Z$. Hence, we implicitly assume that $X_{k}$ can be written as in \eqref{eq_structure_condition}.  For convenience, we write $X_k = g_k(\theta_{k})$ with $\theta_k = (\epsilon_k, \epsilon_{k - 1},\ldots)$.

A nice feature of representation \eqref{eq_structure_condition} is that it allows to give simple, yet very efficient and general dependence conditions. Following \cite{wu_2005}, let $(\epsilon_k')_{k \in \Z}$ be an independent copy of $(\epsilon_k)_{k \in \Z}$ on the same probability space, and define the 'filters' $\theta_{k}^{(l, ')}, \theta_{k}^{(l,*)}$ as
\begin{align}\label{defn_strich_depe}
\theta_{k}^{(l,')} = (\epsilon_{k}, \epsilon_{k - 1},\ldots,\epsilon_{k-l+1},\epsilon_{k - l}',\epsilon_{k - l - 1},\ldots),
\end{align}
and
\begin{align}\label{defn_strich_depe_2}
\theta_{k}^{(l,*)} = \bigl(\epsilon_{k}, \epsilon_{k - 1},\ldots,\epsilon_{k-l+1},\epsilon_{k - l}',\epsilon_{k - l - 1}',\epsilon_{k-l-2}',\ldots\bigr).
\end{align}
We put $\theta_{k}' = \theta_{k}^{(k, ')} = (\epsilon_{k}, \epsilon_{k - 1},\ldots,\epsilon_1,\epsilon_0',\epsilon_{-1},\ldots)$ and ${X}_{k}^{(l,')} = g_k(\theta_{k}^{(l,')})$, in particular, we set ${X}_{k}' ={X}_{k}^{(k,')}$. Similarly, we write
\begin{align*}
\theta_{k}^* = \theta_{k}^{(k,*)} = (\epsilon_{k }, \epsilon_{k - 1},\ldots,\epsilon_1,\epsilon_{0}',\epsilon_{- 1}',\epsilon_{- 2}',\ldots ),
\end{align*}
${X}_{k}^{(l,*)} = g_k(\theta_{k}^{(l,*)})$, and ${X}_{k}^{*} = {X}_{k}^{(k,*)}$. As dependence measures, we may then consider
\begin{align}\label{defn:dep:measure:general}
\vartheta_l'(p) = \sup_{k \in\mathbb{Z}} \|{X}_{k} - {X}_{k}^{(l,\prime)}\|_p, \quad \text{or} \quad \vartheta_l^{\ast}(p) = \sup_{k \in\mathbb{Z}} \|{X}_{k} - {X}_{k}^{(l,\ast)}\|_p.
\end{align}
Observe that if the function $g_k$ satisfies $g_k = g$, that is, it does not depend on $k$, the above simplifies to
\begin{align}\label{defn:dep:measure:bernoulli}
\vartheta_l'(p) = \|{X}_{l} - {X}_{l}^{\prime}\|_p, \quad \vartheta_l^{\ast}(p) = \|{X}_{l} - {X}_{l}^{\ast}\|_p.
\end{align}
In this case, the process $(X_k)_{k \in \Z}$ is typically referred to as (time homogenous) \textit{Bernoulli-shift process}.

Dependence conditions of type \eqref{defn:dep:measure:general}, \eqref{defn:dep:measure:bernoulli} are very general, easy to verify in many prominent cases, and have a long history going back at least to \cite{billingsley_1968}, \cite{ibraginov_1966}, see Section \ref{sec:ex} for a brief account and references for examples.

For any $p \neq 0$, let
\begin{align}\label{boundary}
B(p) = \frac{1}{2} + \frac{p\wedge3}{2p} - \frac{1}{p},
\end{align}
and note that $\lim_{p \to \infty} B(p) = 1/2$. We derive all of our results under the following assumptions.

\begin{ass}\label{ass_main_dependence}
Let $(X_k)_{k\in \Z}$ be stationary such that for $p > 2$, $\ad > 0$ and $\bd > B(p)$, we have
\begin{enumerate}
\item[\Bone]\label{B1} $\|X_k\|_p < \infty$, $\E X_k = 0$,
\item[\Btwo]\label{B2} $\sum_{k = 1}^{\infty}k^{\ad}\vartheta_k^{\ast}(p)$ and $\sum_{k = 1}^{\infty}k^{\bd}\vartheta_k'(p) < \infty$,
\item[\Bthree]\label{B3} $\ss^2 > 0$, where $\ss^2 = \sum_{k \in \Z}\E X_0X_k$.
\end{enumerate}
\end{ass}

Existence of $\ss^2 < \infty$ follows from Lemma \ref{lem_sig_expressions_relations}. It is possible to simplify Assumption \hyperref[B2]{\Btwo} by demanding a (slightly) stronger condition, which is presented in Proposition \ref{prop_unify} below.

\begin{prop}\label{prop_unify}
For $p > 2$ and $\bd > B(p)$, assume
\begin{align*}
\sum_{k = 1}^{\infty} k^{\bd} \sup_{l \geq k}\vartheta_l'(p) < \infty.
\end{align*}
Then there exists $\ad > 0$, such that \hyperref[B2]{\Btwo} is valid. In addition, \hyperref[B2]{\Btwo} may also be replaced by condition
\begin{align*}
\sum_{k = 1}^{\infty}k^{\ad}\sqrt{\sum_{l \geq k} \big(\vartheta_l^{'}(p)\big)^2} < \infty, \quad \sum_{k = 1}^{\infty}k^{\bd}\vartheta_k'(p) < \infty,
\end{align*}
where $\ad > 0$ and $\bd > B(p)$. In particular, if $\bd > 1$, then condition $\sum_{k = 1}^{\infty}k^{\bd}\vartheta_k'(p) < \infty$ alone is sufficient.
\end{prop}

Recall
\begin{align*}
\Delta_n = \sup_{x \in \R} \bigl|\P\bigl(S_n \leq x \sqrt{n \ss^2} \bigr) - \Phi(x)\bigr|,
\end{align*}
where $\ss^2 = \sum_{k \in \Z}\E X_0X_k$. The following is our main result.

\begin{thm}\label{thm_berry}
Grant Assumption \ref{ass_main_dependence}. Then there exists $C>0$, such that
\begin{align*}
\Delta_n \leq C {n^{-(p\wedge 3)/2 + 1}}.
\end{align*}
Constant $C$ only depends on quantities appearing in Assumption \ref{ass_main_dependence}.
\end{thm}

Theorem \ref{thm_berry} provides very general convergence rates under mild conditions. As is demonstrated below in Theorem \ref{thm_lower_bound}, the conditions are essentially sharp. A brief survey of examples satisfying the assumptions is given in Section \ref{sec:ex}. As particular cases, we discuss functions of the dynamical system $Tx = 2x \mod 1$ more detailed in Example \ref{ex_number}, and the left random walk on $GL_d(\R)$ in Example \ref{ex:randomwalk}. Both problems have been studied in the literature for decades, and it appears that Theorem \ref{thm_berry} yields the currently best known results.\\

Let us now turn to the issue of optimality, supplied by the following result.

\begin{thm}\label{thm_lower_bound}
For any $p \geq 3$ and $\bd < 1/2$, there exists a stationary Bernoulli-shift process $(X_k)_{k \in \Z}$ satisfying \hyperref[B1]{\Bone}, \hyperref[B3]{\Bthree} and
\begin{align}
\sum_{k = 1}^{\infty} k^{\bd}\sup_{l \geq k}\|X_l-X_l'\|_p < \infty,
\end{align}
such that for some $\delta > 0$
\begin{align*}
\liminf_{n \to \infty} n^{1/2-\delta} \Delta_n  = \infty.
\end{align*}
\end{thm}

Let us briefly review the indications of Theorems \ref{thm_berry} and \ref{thm_lower_bound}. First, we see that $\bd = 1/2$ is the critical boundary case. If $\sum_{k = 1}^{\infty} k^{\bd} \sup_{l \geq k}\vartheta_l(p) < \infty$, $\bd > 1/2$, then for large enough $p$ (recall $B(p) \to 1/2$), Theorem \ref{thm_berry} in conjunction with Proposition \ref{prop_unify} provides the optimal bound. If $\bd < 1/2$, Theorem \ref{thm_lower_bound} implies that the optimal bound $n^{-1/2}$ is not possible in general. We conclude that the conditions of Theorem \ref{thm_berry} are (essentially) sharp, given sufficiently moments.

\section{Discussion}\label{sec:discussion}

While the result of (essentially) optimality in the previous section is appealing on first glance, the actual proof of Theorem \ref{thm_lower_bound} raises more questions, pointing to a much more complex problem that we will now elaborate on and discuss.

The constructed counter example in Theorem \ref{thm_lower_bound} is a linear process
\begin{align}
X_k = \sum_{j = 0}^{\infty} \alpha_j \epsilon_{k-j}, \quad \E \epsilon_0 = 0, \,\, \E \epsilon_0^2 = 1, \,\, \E |\epsilon_0|^3 < \infty.
\end{align}
One may (but this is not necessary) even take the innovations $(\epsilon_k)_{k \in \Z}$ to be Gaussian.
The actual source of the lower bound is the slow rate of convergence of the variance $n^{-1} \E S_n^2 \to \ss^2$. In this particular counter example, the optimal rate of convergence can be recovered by changing the normalisation from $\sqrt{n}$ to $\sqrt{\E S_n^2}$, which might be somewhat surprising on first sight. It is, however, completely unclear whether this is true more generally or only the case for special (linear) processes. In fact, already the class of linear processes is rich enough to display a variety of not so obvious phenomena, once using $\sqrt{\E S_n^2}$ as normalisation instead of $\sqrt{n s^2}$ in $\Delta_n$:
\begin{itemize}
  \item[(i)] The rate can be faster than $\sqrt{n}$, that is, $\Delta_n = o(n^{-1/2})$.
  \item[(ii)] We may have $\E S_n^2 = o(n)$, but still get $\Delta_n = O(n^{-1/2})$.
\end{itemize}

Let us first elaborate on (i). This has been observed in \cite{HALL_1992115_rate_of_convergence}, where certain linear processes are considered ($\alpha_j = j^{-\alpha} l(j)$ for some slowly varying function $l(\cdot)$). Particular emphasis is put on the case $\alpha = 1/2$. As usual, the basic strategy for assessing linear processes in \cite{HALL_1992115_rate_of_convergence} is the elementary Beveridge and Nelson decomposition (BND) (cf. ~\cite{beveridge_nelson_1981}), where one expresses the sum as
\begin{align}\label{expansion:linear}
\sum_{k = 1}^n X_k/\sqrt{\E S_n^2} = \sum_{k = 1}^N Y_{kN} + R_N,
\end{align}
for some remainder term $R_N$, where $Y_{1N}, Y_{2N}, \ldots,$ and $R_N$ are all independent (often, but not always, $N \thicksim n$ is suitable), see the proof of Lemma \ref{lem_BE_for_linear} for an explicit construction. It should be mentioned though that related, (much) more general martingale decompositions have already appeared earlier in the literature, see for instance ~\cite{gordin_1969} and ~\cite{hannan_1973}. In \cite{HALL_1992115_rate_of_convergence}, Edgeworth expansions are then used to obtain exact remainder terms of first order, implying (in some cases) $\Delta_n = o(n^{-1/2})$. The usage of Edgeworth expansions in \cite{HALL_1992115_rate_of_convergence} requires a non-lattice condition for the innovations $(\epsilon_k)_{k \in \Z}$. We note in passing, that Theorem 5.4 in \cite{petrov_book_1995} may be used as a replacement, resulting in a less precise bound, but the basic result $\Delta_n = o(n^{-1/2})$ is salvaged without any non-lattice condition\footnote{A mild additional assumption for the involved slowly varying function $l(\cdot)$ appears to be necessary.}.

Let us now discuss (ii). This case was brought up by a reviewer, and we use her/his idea of construction. The principle idea is to construct the coefficients $(\alpha_j)_{j \geq 0}$ based on a sequence $(a_j)_{j \geq 0}$, such that ($a_0 = \alpha_0 = 0$ for simplicity)
\begin{align}\label{lin:sequence:condition}
\alpha_1 = a_1, \quad \alpha_j = a_{j} - a_{j-1}, \quad \Big|\sum_{j = 1}^{n} a_j \Big| \to \infty, \quad \sum_{j = 1}^{\infty}|\alpha_j| < \infty.
\end{align}
This construction together with the linearity leads to a cancellation effect, ultimately resulting in $\E S_n^2 = o(n)$. Sequences $(a_j)_{j \geq 1}$ fulfilling \eqref{lin:sequence:condition} are, for instance, $a_j = j^{-\beta}$, $\beta \in (0,1/2)$ or $a_1 = 1/\log(2)$, $a_j = 1/\log(j+1)$ for $j \geq 2$. Elementary computations then reveal $\E S_n^2 \thicksim n^{1- 2 \beta}$ and $\E S_n^2 \thicksim n/\log^2 n$ for the logarithmic case. Proceeding similarly as in the proof of Lemma  \ref{lem_BE_for_linear}, one establishes $\Delta_n = O(n^{-1/2})$. One of the main ingredients is again decomposition \eqref{expansion:linear}, where we point out that the additional factors $n^{\beta}$ (or $\log n$) get cancelled.

The key for both (i) and (ii) is the linearity of $X_k$ and $S_n$, allowing for decomposition \eqref{expansion:linear}. In particular, it is worth noting that for both (i) ($\alpha = 1/2$) and (ii), condition \hyperref[B2]{\Btwo} is violated by the corresponding linear process $X_k$. However, after the linear manipulations, the single summands are independent and thus trivially meet \hyperref[B2]{\Btwo} (apart from the fact that they are no longer stationary).

Already the above considerations indicate that the following question is most likely very difficult to resolve:

\begin{thmx}\label{QB}
What are sharp (weak/other) dependence conditions, such that the optimal rate $n^{-1/2}$ with normalisation $\sqrt{\E S_n^2}$ is reached, that is,
\begin{align}\label{optimal:question:2}
\sup_{x \in \R} \bigl|\P\bigl(S_n \leq x \sqrt{\E S_n^2} \bigr) - \Phi(x)\bigr| \leq C n^{-1/2}
\end{align}
for some constant $C > 0$?
\end{thmx}

The author suspects that there is no universal answer, rather, that some different classes of processes will also display different phase transitions and phenomena. 

\section{Examples}\label{sec:ex}

The literature contains a myriad of examples covered by our setup. Rather than reproducing all these examples, let us mention the following (very small fraction of) references \cite{berkes2014}, \cite{jirak_be_aop_2016},\cite{hoermann_bernoulli_2008}, \cite{wu_2005}, where among others, the following processes are treated: Functions of linear processes, functions of volatility models like Garch, augmented Garch and so on, functions of iterated random models, functions of infinite markov chains, functions of Volterra processes, functions of threshold models, ... . Particularly \cite{wu_2005} contains numerous (additional) examples.\\
The setup also contains many dynamical systems: Due to recent advancements, a connection to H\"{o}lder continuous observables on many (non)-uniformly expanding maps with exponential tails (cf. \cite{korepanov2018}), sub exponential tails (cf. \cite{cuny_quickly_2020}), and also slowly mixing systems (cf. \cite{cuny_dedecker_korepanov_merlevede_2019}) can be made. Among others, this includes Gibbs-Markov maps, Axiom A diffeomorphisms, dispersing billiards, classes of logistic and H\'{e}non maps or the Gauss fraction. In particular, the doubling map 2x mod 1, and cocycles like the left random walk on the general linear group (\cite{jirak_be_aop_2016},\cite{CUNY20181347}), are also within our framework. Since the latter have been studied for decades, we present a more detailed discussion, also in case of functions of linear processes. To the best of my knowledge, the conditions given below in Corollaries \ref{cor:2xmod} and \ref{cor:linear:function} improve upon the currently weakest available in the literature. In light of Theorem \ref{thm_lower_bound}, it is tempting to make the conjecture that the conditions given in Corollaries \ref{cor:2xmod} and \ref{cor:linear:function} are close to being sharp.

\begin{ex}[Sums of the form $\sum f\bigl(t 2^k\bigr)$]\label{ex_number}
For the discussion of this example, we largely follow \cite{jirak_be_aop_2016}. Consider the measure preserving transformation $Tx = 2x \mod 1$ on the probability space $\bigl([0,1],\mathcal{B},\lambdav \bigr)$, with Borel $\sigma$-algebra $\mathcal{B}$ and Lebesgue measure $\lambdav$. Let $U \stackrel{d}{=} \mathrm{Uniform}[0,1]$. Then $T U = \sum_{j = 0}^{\infty} 2^{-j-1}\zeta_{-j}$, where $\zeta_j$ are Bernoulli random variables. The flow $T^k U$ can then be written as $T^k U = \sum_{j = 0}^{\infty} 2^{-j-1} \zeta_{k-j}$, see \cite{ibramigov_1967}. Note that this implies that $T^k U$ can also be represented as an AR(1) process and in particular, that $T^k U$ is a stationary Bernoulli-shift process. The study about the behaviour of $S_n = \sum_{k = 1}^{n} f\bigl(T^k U\bigr)$ for appropriate functions $f$ has a very long history and dates back at least to Kac \cite{kac_1946}. Since then, numerous contributions have been made, see for instance \cite{berkes2014},\cite{billingsley_1999}, \cite{borgne_pene_2005}, \cite{denker_keller_1986}, \cite{dedecker_rio_mean_2008}, \cite{hormann_2009}, \cite{ibramigov_1967}, \cite{ladokhin_1971}, \cite{mcleish_1975}, \cite{moskvin1979local}, \cite{petit_1992}, \cite{postnikov1966ergodic} to name a few. Here, we consider the following class of functions. Recall $B(p)$ in \eqref{boundary}, and let $f$ be a function defined on the unit interval $[0,1]$ such that
\begin{align}\label{eq_ibra_condi}\nonumber
&\int_0^1 f(t) d t = 0, \quad \int_0^1 |f(t)|^p d t < \infty, \quad \text{and}
\\&\int_0^1 t^{-1} |\log(t)|^{\bd} w_p\bigl(f,t\bigr) dt < \infty, \quad \bd > B(p),
\end{align}
where $w_p(f,t)$ denotes a $\mathds{L}^p\bigl([0,1],\lambdav\bigr)$ modulos of continuity of $f$, that is, $w_p \geq 0$ is increasing with
\begin{align*}
\lim_{t \to 0} w_p(f,t) = w_p(f,0) = 0, \quad \int_0^1 \big|f(x+t) - f(x)\big|^p dx \leq w_p^p(t).
\end{align*}
Let $X_k = f\bigl(T^k U\bigr)$. By straightforward computations, we get
\begin{align*}
\big\|X_k - X_k^{\ast} \big\|_p \vee \big\|X_k - X_k' \big\|_p \lesssim w_p\big(f,2^{-k}\big),
\end{align*}

and hence \hyperref[B2]{\Btwo} holds if $\sum_{k = 1}^{\infty} k^{\bd} w_p(f,2^{-k}) < \infty$. This, however, is equivalent with $\int_0^1 t^{-1} |\log(t)|^{\bd} w_p\bigl(f,t\bigr) dt < \infty$, and we obtain the following result. 

\begin{cor}\label{cor:2xmod}
If $\ss^2 > 0$ and \eqref{eq_ibra_condi} holds for $p > 2$, then Theorem \ref{thm_berry} applies.
\end{cor}

If $p = \infty$, then \cite{dedecker_rio_mean_2008} mention that due to a result of \cite{jan:C:thesis}, the decay condition
\begin{align}\label{eq:jan}
\int_0^1 t^{-1} |\log(t)| w_p\bigl(f,t\bigr) dt < \infty
\end{align}
implies the optimal convergence rate $n^{-1/2}$, see also ~\cite{dubois2011} for a related result for Lipschitz-maps. However, for $p \to \infty$, we have $B(p) \to 1/2$, hence condition \eqref{eq_ibra_condi} is strictly weaker even in the special case $p = \infty$.

\end{ex}

\begin{ex}[Left random walk on $GL_d(\R)$]\label{ex:randomwalk}
For the discussion of this example, we largely follow \cite{jirak2020sharp}, where under stronger conditions, the exact transition between a Berry-Esseen bound and Edgeworth expansion is derived in this context. Here, we are only interested in Berry-Esseen bounds, but subject to weaker conditions.\\
Cocycles, in particular the random walk on $GL_d(\R)$, have been heavily investigated in the literature, see e.g. \cite{bougerol_book_1985} and \cite{benoist2016}, \cite{cuny2017}, \cite{CUNY20181347} for some more recent results. We will particularly exploit ideas of \cite{cuny2017}, \cite{CUNY20181347}. As is pointed out in \cite{cuny2019rates_multi}, the argument below also applies to more general cocycles, and consequently also our results.\\
Let $(\varepsilon_k)_{k \geq 0}$ be independent random matrices taking values in $G = GL_d(\R)$, with common distribution $\mu$. Let $A_0 = \mathrm{Id}$, and for every $n \in \N$, $A_n = \prod_{i = 1}^n \varepsilon_i$. Denote with $\|\cdot\|$ the Euclidean norm on $\R^d$. We adopt the usual convention that $\mu$ has a moment of order $q$, if
\begin{align}\label{ex:random:mom}
\int_G \big(\log N(g) \big)^q \mu(d g) < \infty, \quad N(g) = \max\big\{\|g\|, \|g^{-1}\|\big\}.
\end{align}
Let $\mathds{X} = P_{d-1}(\R^d)$ be the projective space of $\R^d\setminus\{0\}$, and write $\overline{x}$ for the projection from $\R^d\setminus\{0\}$ to $\mathds{X}$. We assume that $\mu$ is strongly irreducible and proximal, see \cite{cuny2017} for details. The left random walk of law $\mu$ started at $\overline{x} \in \mathds{X}$ is the Markov chain given by $Y_{0\overline{x}} = \overline{x}$, $Y_{k\overline{x}} = \varepsilon_k Y_{k-1\overline{x}}$ for $k \in \N$. Following the usual setup, we consider the associated random variables $(X_{k\overline{x}})_{k \in \N}$, given by
\begin{align}\label{ex:randomwalk:defn}
X_{k\overline{x}} = h\big(\varepsilon_k, Y_{k-1\overline{x}} \big), \quad h\big(g, \overline{z}\big) = \log \frac{\| g z\|}{\|z\|},
\end{align}
for $g \in G$ and $z \in \R^d\setminus\{0\}$. It follows that, for any $x\in \mathbf{S}^{d-1}$, we have
\begin{align*}
S_{n\overline{x}} = \sum_{k = 1}^n \big(X_{k\overline{x}} - \E X_{k\overline{x}}\big) = \log \big\|A_n x\big\| - \E \log \big\|A_n x\big\|.
\end{align*}
Following \cite{CUNY20181347}, Proposition 3 in \cite{cuny2017} implies that, if $q > 5p/2 + 1/2$, then
\begin{align}\label{ex:prop}
\sum_{k = 1}^{\infty} k^{\bd} \sup_{\overline{x},\overline{y} \in \mathds{X}}\big\|X_{k\overline{x}} - X_{k\overline{y}} \big\|_p < \infty.
\end{align}
In particular, it holds that
\begin{align}\label{ex:s}
\lim_{n \to \infty} n^{-1} \E S_{n\overline{x}}^2 = \ss^2,
\end{align}
where the latter does not depend on $\overline{x} \in \mathds{X}$. We are now in the situation of a \textit{quenched} setup, that is, instead of $X_k$, we have
\begin{align*}
X_{kx}=f_k\big(\epsilon_k, \epsilon_{k-1},\ldots, \epsilon_0,x\big),
\end{align*}
where $x$ is some initial value. Strictly speaking, this quenched setup is not included in our initial setting a priori. However, it is an easy task (but a bit tedious) to check that the results remain equally valid, the crucial points being the validity of \eqref{ex:prop} and \eqref{ex:s}\footnote{In fact, $\ss^2$ may even depend on an initial value $x$, as long as the variance is bounded away from zero.}. Due to Proposition \ref{prop_unify}, we obtain the following result.

\begin{cor}\label{cor:randomwalk}
If $\ss^2 > 0$ and $q > 8$ in \eqref{ex:random:mom}, then Theorem \ref{thm_berry} applies.
\end{cor}

\begin{rem}
Due to Theorem 4.11 (c) in ~\cite{benoist2016}, $s^2 > 0$ as soon as the image of the semigroup $\Gamma_{\mu}$, generated by the support of $\mu$, has unbounded image in the projective group of $GL_d(\R)$.
\end{rem}

Using ideas employed in \cite{CUNY20181347}, the optimal rate was reached very recently in~\cite{cuny:hal-03329189} requiring only four moments, which appears to be the currently best known result. The reduction of moments is possible by additionally exploiting the Markovian nature in some parts of the proof, whereas we simply use the estimate in \eqref{ex:prop}.
\end{ex}

\begin{ex}[Functions of linear process]
Suppose that the sequence $(\alpha_i)_{i \geq 0}$ satisfies $\sum_{i = 0}^{\infty} \alpha_i^2 < \infty$. If $\E|\epsilon_k|^2 < \infty$, then one may show that the linear process
\begin{align*}
Y_k = \sum_{i = 0}^{\infty} \alpha_i \epsilon_{k - i} \quad  \text{exists and is stationary.}
\end{align*}
Let $f$ be a measurable function such that $\E X_k = 0$, where $X_k = f(Y_k)$. If $f$ is H\"{o}lder continuous with regularity $0 < \beta \leq 1$, i.e; $\bigl|f(x) - f(y)\bigr| \leq c |x-y|^{\beta}$, then for any $p \geq 1$
\begin{align*}
\bigl\|X_k - X_k'\bigr\|_p \leq c \alpha_k^{\beta} \bigl\|\epsilon_0\bigr\|_p.
\end{align*}

Now, let $\bd > 1$ and $\ad > 0$ such that $\bd-\ad > 1$. Then, if $\sum_{i = 0}^{\infty} i^{\bd} |\alpha_i|^{\beta} < \infty$, we conclude
$\lim_{i \to \infty} i^{\bd} |\alpha_i|^{\beta} = 0$ and hence
\begin{align*}
\sum_{k = 1}^{\infty}k^{\ad}\sqrt{\sum_{l \geq k} \big(\vartheta_l^{'}(p)\big)^2} \leq \sum_{k = 1}^{\infty}k^{\ad- \bd}\sqrt{\sum_{l \geq k} \big(l^{\bd} |\alpha_l|^{\beta})\big)^2} \lesssim \sum_{k = 1}^{\infty}k^{\ad- \bd} < \infty.
\end{align*}

Due to Proposition \ref{prop_unify}, we thus obtain the following result.

\begin{cor}\label{cor:linear:function}
Let $B(p)$ be as in \eqref{boundary}. If $\ss^2 > 0$, $p \geq 3$, $\beta \in (0,1]$ and either
\begin{itemize}
  \item[(i)]  $\sum_{i = 1}^{\infty} i^{\bd} |\alpha_i|^{\beta} < \infty$, $\bd > 1 \vee B(p)$,
  \item[(ii)] $\sum_{i = 1}^{\infty} i^{\bd} \sup_{k \geq i}|\alpha_k|^{\beta} < \infty$, $\bd > B(p)$,
\end{itemize}
then Theorem \ref{thm_berry} applies.
\end{cor}

Note that if the support of $\epsilon_k$ is bounded and $\beta = 1$, then ~\cite{Rio_1996}, Example 2, requires $\sum_{i = 0}^{\infty} i^2 |\alpha_i| < \infty$, whereas Corollary \ref{cor:linear:function} only requires $\sum_{i = 0}^{\infty} i^{\bd} |\alpha_i| < \infty$ for $\bd > 1$, which can be further relaxed to $\sum_{i = 0}^{\infty} i^{\bd} |\alpha_i| < \infty$, $\bd > 1/2$, if $|\alpha_i|$ is, for instance, monotone in $i$.
\end{ex}

\section{Proof of the main results}\label{sec_proofs}

For the proof of the main results, we refine arguments given in \cite{jirak_be_aop_2016}. For the reader's convenience, we largely provide the whole proof, explicitly indicating which parts are treated as in \cite{jirak_be_aop_2016}. This is done not only for the main proof, but also for almost all the technical lemmas and results.\\
Conceptually, the proof relies on $m$-dependent approximations and delicate conditioning arguments. This means, we
first show Theorem \ref{thm_berry} for certain $m$-dependent sequences (Section \ref{sec:m:dependencies}, Theorem \ref{thm_m_dependent}), and then demonstrate how to derive Theorem \ref{thm_berry} from this result (Section \ref{sec_proof_of_main_theorem}).\\

To simplify the notation in the proofs, we restrict ourselves to the case of (time homogenous) Bernoulli-shift sequences, that is, where
\begin{align*}
X_k = g\bigl(\epsilon_k,\epsilon_{k-1},\ldots \bigr),
\end{align*}
and the function $g$ does not depend on $k$. This requires substantially less notation (indices) throughout the proofs, in particular, \eqref{defn:dep:measure:bernoulli} holds. The more general non-homogenous (but still stationary) case follows from straightforward (notational) adaptations\footnote{This is also true for the non stationary, \textit{quenched} setup.}. This is because the key ingredient, that we require for the proof, is the Bernoulli-shift structure \eqref{eq_structure_condition} in connection with the summability condition \hyperref[B2]{\Btwo}. Whether $g$ depends on $k$ or not is of no relevance in this context.

\subsection{M-dependencies}\label{sec:m:dependencies}

We require some additional notation and definitions. Let $f_m$ be (measurable) functions, and put
\begin{align*}
X_{km} = f_m(\epsilon_k,\ldots,\epsilon_{k-m+1}) \quad \text{for $m \in \N$, $k \in \Z$,}
\end{align*}
where $m = m_n \to \infty$ as $n$ increases. We work under the following conditions.
\begin{ass}\label{ass_dependence}
For $p > 2$ and $\ad, \ss_-^2 > 0$, we have
uniformly in $m$
\begin{enumerate}
\item[\Aone]\label{A1} $\|X_{km}\|_p < \infty$, $\E X_{km} = 0$,
\item[\Atwo]\label{A2} $\sum_{k = 1}^{\infty} k^{\ad}\|{X}_{km} - {X}_{km}^{*} \|_p < \infty$,
\item[\Athree]\label{A3} $\ss_m^2 \geq \ss_-^2 > 0$, where $\ss_m^2 =  \sum_{k \in \Z}\E X_{0m}X_{km} = \sum_{k = -m}^{m}\E X_{0m}X_{km}$.
\end{enumerate}
\end{ass}
Denote with $S_{nm} = \sum_{k = 1}^{n} X_{km}$ and $s_{nm}^2 = n^{-1}\|S_{nm}\|_2^2$ the sample variance. We are now ready to give the main result of this section.

\begin{thm}\label{thm_m_dependent}
Grant Assumption \ref{ass_dependence}, and let $p > 2$. Assume in addition that $m = n^{\mathfrak{m}}$, $0 < \mathfrak{m} < 1$. Then
\begin{align*}
\sup_{x \in \R}\bigl|\P\bigl(S_{nm}/\sqrt{n} \leq x\bigr) - \Phi\bigl(x/s_{nm}\bigr)\bigr| \lesssim  n^{-(p\wedge 3)/2 + 1}.
\end{align*}
\end{thm}

\begin{rem}[Notation]\label{rem:notation}
To simplify notation, we drop the subscript $m$ in $X_{km}$ when the connection is clear. In particular, this will be the case in all proofs whenever Assumption \ref{ass_dependence} is assumed to be valid. Also note that every statement with $X_k$ as in \eqref{eq_structure_condition} naturally contains the special case with $X_{km}$.
\end{rem}

As mentioned earlier, in broad brushes, the proof of Theorem \ref{thm_m_dependent} follows arguments given in \cite{jirak_be_aop_2016}. We deviate, however, at key steps, by employing (somewhat) different technical lemmas, or lemmas where the proofs have been substantially refined. 

These and other new results allow us to improve upon key estimates, respectively weaken the underlying conditions. For the reader's convenience, we largely provide the whole proof, indicating which parts are treated as in \cite{jirak_be_aop_2016}. To this end, let us now recall some notation from \cite{jirak_be_aop_2016}. Introduce the $\sigma$-algebra
\begin{align}\label{defn_sigma_algebra}
\FF_m = \sigma\bigl(\epsilon_{-m+1},\ldots,\epsilon_0,\epsilon_1',\ldots,\epsilon_m',\epsilon_{m+1},\ldots,\epsilon_{2m},\epsilon_{2m+1}',\ldots\bigr),
\end{align}
where we recall that $(\epsilon_k)_{k \in \Z}$ and $(\epsilon_k')_{k \in \Z}$ are mutually independent, identically distributed random sequences. For a $\sigma$-algebra $\mathcal{H}$, we write $\P_{\mathcal{H}}(\cdot)$ for the conditional law and $\E_{\mathcal{H}}$ for the conditional expectation. Put
\begin{align*}
S_{|m}^{(1)} = \sum_{k = 1}^n  \big(X_{km} - \E_{\FF_m} X_{km}\big) \quad \text{and} \quad S_{|m}^{(2)} = \sum_{k = 1}^n \E_{\FF_m} X_{km},
\end{align*}
yielding the decomposition
\begin{align*}
S_{nm} = \sum_{k = 1}^n X_{km}  = S_{|m}^{(1)} + S_{|m}^{(2)}.
\end{align*}
To avoid any notational problems, we put $X_{km} = 0$ for $k \not \in \{1,\ldots,n\}$. Let $n = 2(N-1)m + m'$, where $N,m$ are chosen such that $C m \leq m' \leq m$ and $C > 0$ is an absolute constant, independent of $m,n$. For $1 \leq j \leq N$, we construct the block random variables
\begin{align}\label{defn_U_R}
U_j = \sum_{k = (2j-2)m + 1}^{(2j - 1)m} \big(X_{km} - \E_{\FF_m} X_{km}\big) \quad \text{and} \quad R_j =  \sum_{k = (2j-1)m + 1}^{2jm} \big(X_{km} - \E_{\FF_m} X_{km}\big),
\end{align}
and put $Y_j^{(1)} = U_j + R_j$, hence $S_{|m}^{(1)} = \sum_{j = 1}^{N} Y_j^{(1)}$. Note that by construction of the blocks, $Y_j^{(1)}$, $j = 1,\ldots,N$ are independent random variables under the conditional probability measure $\P_{\FF_m}(\cdot)$, and are identically distributed at least for $j = 1,\ldots,N-1$ under $\P$.
We also put $Y_1^{(2)} = \sum_{k = 1}^{m} \E_{\FF_m} X_{km}$ and $Y_j^{(2)} = \sum_{k = (j-1)m + 1}^{(j+1)m} \E_{\FF_m} X_{km}$ for $j = 2,\ldots,N$. Note that $Y_j^{(2)}$, $j = 1,\ldots,N$ is a sequence of independent random variables. The following partial and conditional variances are relevant for the proofs:
\begin{align}\nonumber
\sigma_{j|m}^2 &= \frac{1}{2m} \E_{\FF_m} \bigl(Y_j^{(1)}\bigr)^2, \qquad \sigma_j^2 = \E \sigma_{j|m}^2,\\ \nonumber
\sigma_{|m}^2 &= \frac{1}{n}\E_{\FF_m} (S_{|m}^{(1)})^2 = \frac{1}{N -1 + m'/2m} \sum_{j = 1}^N \sigma_{j|m}^2,\\ \nonumber
\overline{\sigma}_m^2 &= \E \sigma_{|m}^2 = \frac{1}{N-1 + m'/2m} \sum_{j = 1}^N \sigma_j^2,\\ \nonumber
\widehat{\sigma}_m^2 &= \frac{1}{2m}\sum_{k = 1}^m\sum_{l = 1}^m \E X_{km} X_{lm}.
\end{align}
Obviously, these quantities are all closely connected. Note that $\sigma_i^2 = \sigma_j^2$ for $1 \leq i,j \leq N-1$, but $\sigma_1^2 \neq \sigma_{N}^2$ in general. Moreover, we have the (formal) equation
\begin{align}\label{decomp_sigma}
2m \widehat{\sigma}_m^2 = m \ss_m^2 - \sum_{k \in \Z} \big(m \wedge |k|\big) \E X_{0m} X_{km}.
\end{align}
The above relation is important, since Lemma \ref{lem_sig_expressions_relations} yields that under Assumption \ref{ass_dependence} we have $2\widehat{\sigma}_m^2 = \ss_m^2 + \oo(1)$. Moreover, Lemma \ref{lem_sig_expansion} implies $\sigma_j^2 = \widehat{\sigma}_m^2 + \oo(1)$ for $1 \leq j \leq N-1$. We conclude that
\begin{align}\label{eq_sigma_bigger_zero}
\sigma_j^2 = \ss_m^2/2 + \oo(1) > 0, \quad \text{for sufficiently large $m$.}
\end{align}
The same is true for $\sigma_N^2$, since $m' \geq C m$. All in all, we see that we are not facing any degeneracy problems for the partial variances $\sigma_j^2$, $1 \leq j \leq N$ subject to Assumption \ref{ass_dependence}. For the second part $S_{|m}^{(2)}$, we introduce $\overline{\varsigma}_{m}^2 = n^{-1}\bigl\|S_{|m}^{(2)}\bigr\|_2^2$. One readily derives via simple conditioning arguments
\begin{align}\label{defn_snm}
\ss_{nm}^2 = n^{-1}\bigl\|S_{nm}\bigr\|_2^2 = n^{-1} \bigl\|S_{|m}^{(1)}\bigr\|_2^2 +  n^{-1} \bigl\|S_{|m}^{(2)} \bigr\|_2^2 = \overline{\sigma}_{m}^2 + \overline{\varsigma}_{m}^2.
\end{align}

We require some additional notation. Let $(\epsilon_k'')_{k \in \Z}$ be independent copies of $(\epsilon_k)_{k \in \Z}$ and $(\epsilon_k')_{k \in \Z}$. For $l \leq k$, we then define $X_{km}^{(l,'')}, X_{km}^{(l,**)}, X_{km}^{''}, X_{km}^{**}$ in analogy to $X_k^{(l,')}, X_k^{(l,*)}, X_k^{'}, X_k^{*}$, replacing every $\epsilon_k'$ with $\epsilon_k''$ at all corresponding places. For $k \geq 0$, we also introduce the $\sigma$-algebras
\begin{align*}
\mathcal{E}_k' = \sigma\bigl(\epsilon_j, \, j \leq k \, \text{and} \, j \neq 0, \, \epsilon_0'\bigr)\quad \text{and} \quad \mathcal{E}_k^* = \sigma\bigl(\epsilon_j, \, 1 \leq j \leq k \, \text{and} \,  \epsilon_i', \, i \leq 0\bigr).
\end{align*}
Similarly, we define $\mathcal{E}_k''$ and $\mathcal{E}_k^{**}$. Throughout the proofs, we make the following conventions.
\begin{description}
\item[(i)] We do not distinguish between $N$ and $N-1 + m'/2m$, since the difference is of no relevance for the proofs. We use $N$ for both expressions.
\item[(ii)] The abbreviations $I$, $II$, $III$, $\ldots$, for expressions (possible with some additional indices) vary from proof to proof.
\item[(iii)] If there is no confusion, we put $Y_j = (2m)^{-1/2}Y_j^{(1)}$ for $j = 1,\ldots,N$ to lighten the notation, particularly in part $\textbf{A}$.
\item[(iv)] $C > 0$ denotes a constant, which may vary from line to line. In addition, it may depend (simultaneously) on any of the quantities appearing in Assumption \ref{ass_dependence} (resp. Assumption \ref{ass_main_dependence}).
\item[(v)] Remark \ref{rem:notation} is in action.
\end{description}

For the proofs, we first establish key preliminary results as lemmas in Section \ref{sec_main_lemmas} below. We then provide the proof of Theorem \ref{thm_m_dependent} in Section \ref{sec:proof:thm:m}. Theorem \ref{thm_berry} is based on Theorem \ref{thm_m_dependent}, the proof is given in Section \ref{sec_proof_of_main_theorem}. Finally, the results for the lower bound are given in Section \ref{sec:proof:lower:bound}.

\subsection{Main Lemmas}\label{sec_main_lemmas}

We frequently use the following Lemma in connection with Lemma \ref{lem_star_gives_strip}, which is essentially a restatement of Theorem 1 in \cite{sipwu}, adapted to our setting.
\begin{lem}\label{lem_wu_original}
Put $p' = \min\{p,2\}$. If $\sum_{k = 1}^{\infty}\|X_k-X_k'\|_p < \infty$, then
\begin{align*}
\bigl\|X_1 + \ldots + X_n\bigr\|_{p} \lesssim n^{1/p'}.
\end{align*}
\end{lem}

The next result controls the remainder $R_j$.

\begin{lem}\label{lem_bound_R1}
Grant Assumption \ref{ass_dependence}. Then $\|R_j\|_p < \infty$ for $j = 1,\ldots,N$, where $R_j$ is defined in \eqref{defn_U_R}.
\end{lem}

\begin{proof}[Proof of Lemma \ref{lem_bound_R1}]
Recall Remark \ref{rem:notation}. Without loss of generality, we can assume that $j = 1$ due to $m \thicksim m'$. Since $X_k - \E_{\FF_m} X_k \stackrel{d}{=} \E_{\FF_m}\bigl[X_k^{(k-m,*)} - X_k\bigr]$ for $m+1 \le k \leq 2m$, we have
\begin{align*}
&\Bigl\|\sum_{k = m+1}^{2m} \big(X_k - \E_{\FF_m}X_k\big)\Bigr\|_p = \Bigl\|\sum_{k = m+1}^{2m} \E_{\FF_m}\bigl[X_k^{(k-m,*)} - X_k\bigr]\Bigr\|_p \\&\leq \sum_{k = m+1}^{2m}\bigl\| X_k^{(k-m,*)} - X_k\bigr\|_p \leq \sum_{k = 1}^{\infty}\bigl\|X_k^* - X_k\bigr\|_p < \infty.
\end{align*}
\end{proof}

The following lemma establishes a simple connection between $\|X_k - X_k'\|_p$ and $\|X_k - X_k^{\ast}\|_p$, which we will frequently use without mentioning it explicitly any further.

\begin{lem}\label{lem_star_gives_strip}
Assume $\sum_{k = 1}^{\infty}k^{\ad} \|X_k - X_k^{\ast}\|_p < \infty$. Then $\sum_{k = 1}^{\infty}k^{\ad} \|X_k - X_k'\|_p < \infty$.
\end{lem}

\begin{proof}[Proof of Lemma \ref{lem_star_gives_strip}]
Observe that we have the inequality
\begin{align*}
\bigl\|X_k - X_k'\bigr\|_p &\leq \bigl\|X_k - X_k^*\bigr\|_p + \bigl\|X_k' - X_k^*\bigr\|_p\\&= \bigl\|X_k - X_k^*\bigr\|_p + \bigl\|X_{k+1} - X_{k+1}^*\bigr\|_p,
\end{align*}
hence the claim readily follows.
\end{proof}

\begin{lem}\label{lem_sig_expressions_relations}
Assume $\sum_{k = 1}^{\infty} k^{\cd} \vartheta_{k}(2) < \infty$ for $\cd > 0$ and $\E X_k = 0$. Then $\sum_{k = 1}^{\infty}|\E X_0 X_k|  < \infty$ and
\begin{align*}
n^{-1}\Bigl\|\sum_{k = 1}^n X_k \Bigr\|_2^2 = \sum_{k \in \Z} \E X_0 X_k + \OO\bigl(n^{-\cd}\bigr).
\end{align*}
In particular, $\widehat{\sigma}_m^2 = \ss_m^2/2 + \OO(m^{-\cd})$ and $\widehat{\sigma}_l^2 = \ss_m^2/2 + \oo\bigl(1\bigr)$ as $l \to m$.
\end{lem}

\begin{proof}[Proof of Lemma \ref{lem_sig_expressions_relations}]
For a random variable $X$, define the projections $\Pro_i(X)= \E_{\mathcal{E}_i} X -  \E_{\mathcal{E}_{i-1}} X$. Note that for $i \leq k$
\begin{align}\label{eq:lem:sig:expression:relation:1}
\Pro_i(X_k)= \E_{\mathcal{E}_i}\big[X_k - X_k^{(k-i,')}\big], \quad \big\|\Pro_i(X_k) \big\|_2 \leq \vartheta_{k-i}(2).
\end{align}

We then have the well-known decomposition $X_k = \sum_{i = -\infty}^k \Pro_i(X_k)$. Existence of the sum on the right-hand-side follows for instance from the triangle inequality and \eqref{eq:lem:sig:expression:relation:1}. Alternatively, one may also employ martingale arguments. From the orthogonality of the projections, we get
\begin{align*}
\E X_0X_k = \sum_{i = - \infty}^{0} \E \Pro_i(X_0) \Pro_i(X_k).
\end{align*}
Hence an application of Cauchy-Schwarz and \eqref{eq:lem:sig:expression:relation:1} yields 
\begin{align*}
\sum_{k \in \Z} \bigl|\E X_0 X_k \bigr| &\leq \sum_{i = - \infty}^{0} \bigl\|\Pro_i(X_0)\bigr\|_2 \sum_{k > i-1} \bigl\|\Pro_i(X_k)\bigr\|_2 \leq \Bigl(\sum_{i = 0}^{\infty} \vartheta_{i}(2)\Bigr)^2 < \infty.
\end{align*}
Similarly, we derive

\begin{align*}
\sum_{k = 1}^{\infty} (k \wedge n) \bigl|\E X_0 X_k \bigr| &\lesssim \sum_{i = - \infty}^{0} \bigl\|\Pro_i(X_0)\bigr\|_2 \sum_{k = 1}^n k  \bigl\|\Pro_i(X_k)\bigr\|_2 + n \sum_{i = - \infty}^{0} \bigl\|\Pro_i(X_0)\bigr\|_2 \sum_{k > n}  \bigl\|\Pro_i(X_k)\bigr\|_2
\\&\lesssim n^{1-\cd} \sum_{i = - \infty}^{0} \vartheta_{-i}(2) \sum_{k = 1}^n (k-i)^{\cd}  \vartheta_{k-i}(2) + n^{1 - \cd} \sum_{i = - \infty}^{0} \vartheta_{-i}(2) \sum_{k > n} (k-i)^{\cd}   \vartheta_{k-i}(2) \\&\lesssim n^{1- \cd}.
\end{align*}

Next, observe
\begin{align}\label{eq_lem_sig_expressions_relations_2.1}
\Bigl\|\sum_{k = 1}^n X_k \Bigr\|_2^2 = n \sum_{k \in \Z} \E X_0 X_k - \sum_{k \in \Z} (n \wedge |k|) \E X_0 X_k,
\end{align}
and using the above, we conclude 
\begin{align*}
n^{-1}\Bigl\|\sum_{k = 1}^n X_k \Bigr\|_2^2 = \sum_{k \in \Z} \E X_0 X_k + \OO\bigl(n^{-\cd}\bigr).
\end{align*}
Hence $\widehat{\sigma}_m^2 = \ss_m^2/2 + \OO\bigl(m^{-\cd}\bigr)$, and $\widehat{\sigma}_l^2 = \ss_m^2/2 + \oo\bigl(1\bigr)$, as $l \to m$, readily follows.
\end{proof}

The following Lemma is a key result. At some instances, we employ the same reasoning as in \cite{jirak_be_aop_2016}, at others, we argue rather differently.

\begin{lem}\label{lem_sig_expansion}
Grant Assumption \ref{ass_dependence}. Then there exists a $\delta > 0$ such that
\begin{itemize}
\item[(i)]$\bigl\|\sigma_{j|m}^2 - \sigma_j^2\bigr\|_{p/2} \lesssim \bigl\|\sigma_{j|m}^2 - \widehat{\sigma}_m^2\bigr\|_{p/2}+m^{-1/2- \delta} \lesssim m^{-1/2- \delta}$ \\for $1 \leq j \leq N$,
\item[(ii)] $\sigma_j^2 = \widehat{\sigma}_m^2 + \OO\bigl(m^{-1/2 - \delta}\bigr)$ for $1 \leq j \leq N$,
\item[(iii)] $\bigl\|\sigma_{|m}^2 - \overline{\sigma}_m^2\bigr\|_{p/2} \lesssim n^{-1/2 - \delta} N^{2/p-1/2 + \delta}$.
\end{itemize}
\end{lem}

\begin{proof}[Proof of Lemma \ref{lem_sig_expansion}]
Recall Remark \ref{rem:notation}. We first show (i). Without loss of generality, we may assume $j = 1$, since $m \thicksim m'$. To lighten the notation, we use $R_1 = R_1^{(1)}$. We first establish $\bigl\|\sigma_{j|m}^2 - \widehat{\sigma}_m^2\bigr\|_{p/2} \lesssim m^{-1/2 - \delta}$, $\delta > 0$. Observe
\begin{align*}
2m\bigl(\sigma_{1|m}^2 - \widehat{\sigma}_m^2\bigr)= \E_{\FF_m}\Bigl[\Bigl(\sum_{k = 1}^m \bigl(X_k^{(**)} + (X_k - X_k^{(**)}) - \E_{\FF_m} X_k\bigr) + R_1 \Bigr)^2\Bigr] - 2m  \widehat{\sigma}_m^2.
\end{align*}
Computing the square in the first expression, we obtain a sum of squared terms and of cross terms. We first consider the cross terms, which are
\begin{align*}
&2\sum_{k = 1}^m \sum_{l = 1}^m \E_{\FF_m}\Bigl[\E_{\FF_m}\bigl[X_k^{(**)}(X_l - X_l^{(**)})\bigr] + \E_{\FF_m}[X_k^{(**)}]\E_{\FF_m}[X_l] + \E_{\FF_m}[X_k](X_l - X_l^{(**)})\Bigr] \\&+ 2 \sum_{k = 1}^m \E_{\FF_m}\Bigl[R_1 X_k^{(**)} +R_1(X_k - X_k^{(**)}) + R_1\E_{\FF_m}[X_k] \Bigr] \\& \stackrel{def}{=} I_m + II_m +III_m + IV_m + V_m + VI_m.
\end{align*}
Below, we separately develop bounds for each of these terms.\\
{\bf Case $I_m$:} We have
\begin{align*}
I_m/2 &= \sum_{l = 1}^m \sum_{k = l}^m \E_{\FF_m}\bigl[X_k^{(**)}(X_l - X_l^{(**)})\bigr] + \sum_{l = 1}^m \sum_{k = 1}^{l-1}\E_{\FF_m}\bigl[X_k^{(**)}(X_l - X_l^{(**)})\bigr] \\&= \sum_{l = 1}^m \sum_{k = l}^m \E_{\FF_m}\Bigl[(X_l - X_l^{(**)})\E\bigl[X_k^{(**)}\bigl|\sigma\bigl(\FF_m,\mathcal{E}_l,\mathcal{E}_l^{(**)} \bigr) \bigr] \Bigr] \\&+ \sum_{l = 1}^m \sum_{k = 1}^{l-1}\E_{\FF_m}\bigl[X_k^{(**)}(X_l - X_l^{(**)})\bigr].
\end{align*}
Observe that we have the (distributional) identities
\begin{align*}
\E\bigl[X_k^{(**)}\bigl|\sigma\bigl(\FF_m,\mathcal{E}_l,\mathcal{E}_l^{(**)} \bigr) \bigr] \stackrel{d}{=} \E\bigl[X_k\bigl|\mathcal{E}_l\bigr] = \E\bigl[X_k - X_k^{(k-l,*)}\bigl|\mathcal{E}_l\bigr].
\end{align*}
Hence by Cauchy-Schwarz (with respect to $\E_{\FF_m}$) and Jensen's inequality
\begin{align*}
\bigl\|I_m\bigr\|_{p/2} &\lesssim \sum_{l = 1}^m \sum_{k = l}^m \bigl\|X_l - X_l^{(**)} \bigr\|_{p}\bigl\|X_k - X_k^{(k-l,*)} \bigr\|_p + \sum_{l = 1}^m \bigl\|\sum_{k = 1}^{l-1}X_k^{(**)}\bigr\|_{p}\bigl\|X_l - X_l^{(**)}\bigr\|_{p}
\end{align*}
which, by using Lemma \ref{lem_wu_original}, is further bounded by
\begin{align*}
\lesssim \Bigl(\sum_{l = 1}^{m}\bigl\|X_l-X_l^*\bigr\|_p\Bigr)^2 + \sum_{l = 1}^{m}\sqrt{l}\bigl\|X_l-X_l^*\bigr\|_p.
\end{align*}
Due to \hyperref[A2]{\Atwo}, we thus obtain
\begin{align}
\bigl\|I_m\bigr\|_{p/2} &\lesssim  m^{1/2 - \ad}.
\end{align}
{\bf Case $II_m$:} Since $\E_{\FF_m}X_k^{(**)} =\E X_k = 0$, we get $II_m = 0$.\\
{\bf Case $III_m$:} First, Jensen's inequality implies
\begin{align}\label{eq_lem_sig_expansion_3}
\bigl\|\E_{\FF_m}X_l\bigr\|_p = \bigl\|\E_{\FF_m}\bigl[X_l - X_l^{**}\bigr]\bigr\|_p \leq \bigl\|X_l - X_l^{**} \bigr\|_p.
\end{align}
Cauchy-Schwarz (with respect to $\E_{\FF_m}$), Jensen's inequality and \hyperref[A2]{\Atwo} now yield
\begin{align}
\bigl\|III_m\bigr\|_{p/2} \lesssim \Bigl(\sum_{l = 1}^m \bigl\|X_l - X_l^*\bigr\|_p\Bigr)^2 < \infty.
\end{align}
{\bf Case $IV_m$:} Let $\tau > 0$ and put $\tau_m = m^{\tau}$. Then by the triangle inequality
\begin{align}\label{eq_lem_sig_expansion_4}
\Bigl\|\sum_{k = 1}^m \E_{\FF_m}\bigl[R_1 X_k^{(**)}\bigr] \Bigr\|_{p/2} &\leq \Bigl\|\sum_{k = m - \tau_m + 1}^{m} \E_{\FF_m}\bigl[R_1 X_k^{(**)}\bigr] \Bigr\|_{p/2} + \Bigl\|\sum_{k = 1}^{m - \tau_m} \E_{\FF_m}\bigl[X_k^{(**)}R_1\bigr] \Bigr\|_{p/2}.
\end{align}
The triangle and Jensen's inequality, Cauchy-Schwarz and $\|R_1\|_p < \infty$ by Lemma \ref{lem_bound_R1}, yield
\begin{align}\label{eq_lem_sig_expansion_5}
\Bigl\|\sum_{k = m - \tau_m + 1}^{m} \E_{\FF_m}\bigl[R_1 X_k^{(**)}\bigr] \Bigr\|_{p/2} \leq \sum_{k = m - \tau_m + 1}^{m} \bigl\|X_k\bigr\|_p \bigl\|R_1\bigr\|_p \lesssim \tau_m.
\end{align}
The second term in \eqref{eq_lem_sig_expansion_4} is more complicated. Recall $R_1 =  \sum_{k = m + 1}^{2m} \big(X_k - \E_{\FF_m} X_k\big)$. Observe that $X_k^{(**)}$ and $X_l^{(l-m+\tau_m,*)}$ are independent for $1 \leq k \leq m - \tau_m$ and $m+1\leq l \leq 2m$. Since $\E_{\FF_m} X_k^{(**)} = 0$, we get
\begin{align*}
\sum_{k = 1}^{m- \tau_m} \E_{\FF_m} R_1 X_k^{(**)} &= \sum_{k = 1}^{m- \tau_m}\sum_{l = m+1}^{2m} \E_{\FF_m}\bigl[X_k^{(**)}X_l\bigr],
\end{align*}
and hence
\begin{align*}
\sum_{k = 1}^{m- \tau_m} \E_{\FF_m}\bigl[R_1 X_k^{(**)}\bigr] &= \sum_{k = 1}^{m- \tau_m}\sum_{l = m+1}^{2m} \E_{\FF_m}\bigl[X_k^{(**)}(X_l - X_l^{(l-m + \tau_m,*)} + X_l^{(l-m + \tau_m,*)})\bigr] \\&= \E_{\FF_m}\Bigl[\sum_{k = 1}^{m - \tau_m} X_k^{(**)}\sum_{l = m+1}^{2m}\bigl(X_l - X_l^{(l-m + \tau_m,*)}\bigr)\Bigr],
\end{align*}
where we used $\E_{\FF_m}\bigl[X_k^{(**)}X_l^{(l-m + \tau_m,*)}\bigr] = 0$. 
Then Cauchy-Schwarz, the triangle and Jensen's inequality together with Lemma \ref{lem_wu_original} and \hyperref[A2]{\Atwo} yield
\begin{align}\nonumber \label{eq_lem_sig_expansion_6}
\bigl\|\sum_{k = 1}^{m- \tau_m} \E_{\FF_m}\bigl[R_1 X_k^{(**)}\bigr]\bigr\|_{p/2} &\lesssim \bigl\|\sum_{k = 1}^{m - \tau_m} X_k \bigr\|_p \sum_{l = \tau_m}^{\infty} \bigl\|X_l-X_l^*\bigr\|_p\\&\lesssim \sqrt{m} \tau_m^{-\ad} \sum_{l = \tau_m}^{\infty}l^{\ad}\bigl\|X_l-X_l^*\bigr\|_p \lesssim  \sqrt{m} \tau_m^{-\ad}.
\end{align}
Piecing both bounds \eqref{eq_lem_sig_expansion_5} and \eqref{eq_lem_sig_expansion_6} together and selecting $\tau > 0$ sufficiently small, it follows that
\begin{align}\label{eq_lem_sig_expansion_7_ad positive}
\bigl\|IV_m\bigr\|_{p/2} &\lesssim \tau_m + \sqrt{m}\tau_m^{-\ad} \lesssim m^{1/2 - \ad \tau}, \quad \ad > 0.
\end{align}

{\bf Case $V_m$:} Cauchy-Schwarz, Jensen's inequality, Lemma \ref{lem_bound_R1} and \hyperref[A2]{\Atwo} yield
\begin{align*}
\bigl\|V_m\bigr\|_{p/2} &\lesssim \sum_{k = 1}^m \bigl\|X_k - X_k^{(**)}\bigr\|_p \bigl\|R_1\bigr\|_p < \infty.
\end{align*}
{\bf Case $VI_m$:} Proceeding as above and using \eqref{eq_lem_sig_expansion_3}, it follows that
$\bigl\|VI_m\bigr\|_{p/2} < \infty$.\\
\\
It thus remains to deal with the squared terms, which are
\begin{align*}
&\sum_{k = 1}^m\sum_{l = 1}^m\E_{\FF_m}\bigl[X_k^{(**)}X_l^{(**)} + (X_k - X_k^{(**)})(X_l - X_l^{(**)}) + \E_{\FF_m}[X_k]\E_{\FF_m}[X_l]\bigr] + \E_{\FF_m}\bigl[R_1^2\bigr] \\&= 2m\widehat{\sigma}_m^2 + VII_m + VIII_m + IX_m.
\end{align*}
However, using the results from the previous computations and Lemma \ref{lem_bound_R1}, one readily deduces that
\begin{align}
\bigl\|VII_m\bigr\|_{p/2}, \, \bigl\|VIII_m\bigr\|_{p/2}, \, \bigl\|IX_m\bigr\|_{p/2} < \infty.
\end{align}
Piecing everything together, we have established that for sufficiently small $\tau > 0$
\begin{align}\nonumber
\bigl\|\sigma_{j|m}^2 - \widehat{\sigma}_m^2\bigr\|_{p/2} \lesssim m^{-1/2 - \tau \ad} + m^{-1} \lesssim m^{-1/2 - \tau \ad}.
\end{align}
Hence we get that for $\delta = \tau \ad > 0$ 
\begin{align}
\bigl\|\sigma_{j|m}^2 - \widehat{\sigma}_m^2\bigr\|_{p/2} \lesssim m^{-1/2 - \delta}.
\end{align}
Moreover, from the above arguments, one also obtains the same bound for $\|\sigma_{j|m}^2 - \sigma_j^2\|_{p/2}$, or directly using $|\sigma_j^2 - \widehat{\sigma}_m^2| \lesssim  m^{-1/2 - \delta}$. 
In either case, (i) and (ii) follow. We now treat (iii). Since $(Y_j^{(1)})_{1 \leq j \leq N}$ are independent under $\P_{\FF_m}$, we have
\begin{align}\label{eq_lem_sig_expansion_7}
\sigma_{|m}^2 = N^{-1} \sum_{j = 1}^N \sigma_{j|m}^2.
\end{align}
Let $\mathcal{I} = \{1,3,5,\ldots,N\}$ and $\mathcal{J} = \{2,4,6,\ldots,N\}$ such that $\mathcal{I} \uplus \mathcal{J} = \{1,2,\ldots,N\}$. Then
\begin{align*}
\Bigl\|\sum_{j = 1}^N \big(\sigma_{j|m}^2 - \sigma_j^2\big) \Bigr\|_{p/2} \leq \Bigl\|\sum_{j \in \mathcal{I}} \big(\sigma_{j|m}^2 - \sigma_j^2\big)\Bigr\|_{p/2} + \Bigl\|\sum_{j \in \mathcal{J}} \big(\sigma_{j|m}^2 - \sigma_j^2\big) \Bigr\|_{p/2}.
\end{align*}
Note that $(\sigma_{j|m}^2)_{j \in \mathcal{I}}$ is a sequence of independent random variables, and the same is true for $(\sigma_{j|m}^2)_{j \in \mathcal{J}}$. Then by Lemma \ref{lem_wu_original}, we have
\begin{align*}
\Bigl\|\sum_{j = 1}^N \big(\sigma_{j|m}^2 - \sigma_j^2\big) \Bigr\|_{p/2} \lesssim N^{2/p} \bigl\|\sigma_{j|m}^2 - \sigma_j^2 \bigr\|_{p/2}, \quad \text{for $p \in (2,3]$,}
\end{align*}
which by (i) is of magnitude $\OO\bigl(N^{2/p} m^{-1/2 - \delta}\bigr)$. Hence we conclude from \eqref{eq_lem_sig_expansion_7} the bound
\begin{align*}
\bigl\|\sigma_{|m}^2 - \overline{\sigma}_m^2\bigr\|_{p/2} \lesssim n^{-1/2 - \delta} N^{2/p-1/2 + \delta}.
\end{align*}
This completes the proof.
\end{proof}

For our next result, Lemma \ref{lem_sig_lower_bound} below, we require the following preliminary result (Lemma 6.3 in \cite{chen2016:aos}).
\begin{lem}\label{lem_exp_bound_wu}
Let $(U_k)_{1\leq k\leq n}$ be independent, non-negative variables with $\|U_k\|_q < \infty$, where $1 \leq q \leq 2$. Then for any $0 < u < \sum_{k = 1}^n \E U_k$, we have the bound
\begin{align*}
\P\Bigl(\sum_{k = 1}^n U_k \leq \sum_{k = 1}^{n} \E U_k - u \Bigr) \leq \exp\Bigl(-\frac{q-1}{4} \frac{u^{q/(q-1)}}{\bigl(\sum_{k = 1}^n \E|U_k|^q\bigr)^{1/(q-1)}}\Bigr).
\end{align*}
\end{lem}

We are now ready to establish the following lemma.

\begin{lem}\label{lem_sig_lower_bound}
Grant Assumption \ref{ass_dependence}. Then there exists a constant $C > 0$, only depending on $p$, such that
\begin{align*}
\P\Bigl(\frac{1}{N}\sum_{j = 1}^N \sigma_{j|m}^2 \leq  \ss_m^2/8\Bigr) \lesssim \exp\Bigl(-C\,N\Bigr).
\end{align*}
\end{lem}

\begin{proof}[Proof of Lemma \ref{lem_sig_lower_bound}]
Let $\mathcal{I} = \{1,3,5,\ldots,N\}$ and $\mathcal{J} = \{2,4,6,\ldots,N\}$ such that $\mathcal{I} \uplus \mathcal{J} = \{1,2,\ldots,N\}$. Then
\begin{align*}
\P\Bigl(\frac{1}{N}\sum_{j = 1}^N \sigma_{j|m}^2 \leq  \ss_m^2/8\Bigr) \leq \P\Bigl(\frac{1}{N}\sum_{j\in \mathcal{I}} \sigma_{j|m}^2 \leq \ss_m^2/8 \Bigr).
\end{align*}
Recall that $(\sigma_{j|m}^2)_{j \in \mathcal{I}}$ is a sequence of independent random variables. Moreover, Lemma \ref{lem_sig_expressions_relations}, Lemma \ref{lem_sig_expansion} and the triangle inequality yield
\begin{align}
\sum_{j\in \mathcal{I}} \bigl\|\sigma_{j|m}^2\bigr\|_{p/2}^{p/2} \leq \sum_{j\in \mathcal{I}}\bigl(\ss_m^2/2 + \oo(1)\bigr)^{p/2} \lesssim \bigl|\mathcal{I}\bigr| \lesssim N,
\end{align}
where $|\mathcal{I}|$ denotes the cardinality of $\mathcal{I}$. Similarly, we also have
\begin{align}
\sum_{j\in \mathcal{I}} \E \sigma_{j|m}^2 \geq \sum_{j\in \mathcal{I}}\bigl(\ss_m^2/2 - \oo(1)\bigr) \geq N \ss_m^2/4 - \oo\bigl(N\bigr).
\end{align}
Hence, selecting $u = \sum_{j\in \mathcal{I}} \E \sigma_{j|m}^2 - N\ss_m^2/8$, it follows from the above that $u \gtrsim N$. Setting $q = p/2 > 1$, an application of Lemma \ref{lem_exp_bound_wu} then yields
\begin{align*}
\P\Bigl(\frac{1}{N}\sum_{j\in \mathcal{I}} \sigma_{j|m}^2 \leq \ss_m^2/8 \Bigr) \lesssim \exp\Bigl(-C N^{\frac{q}{q-1} - \frac{1}{q-1}}\Bigr) \lesssim \exp\Bigl(-C N \Bigr),
\end{align*}
where $C > 0$ is an absolute constant.
\end{proof}

\subsection{Proof of Theorem \ref{thm_m_dependent}}\label{sec:proof:thm:m}

We are now ready for the proof of Theorem \ref{thm_m_dependent}, which uses the following decomposition. Let $Z_1,Z_2$ be independent, standard Gaussian random variables. Then
\begin{align*}
&\sup_{x \in \R}\bigl|\P\bigl(S_{nm}/\sqrt{n} \leq x\bigr) - \Phi\bigl(x/s_{nm}\bigr)\bigr| \\&= \sup_{x \in \R}\Bigl|\P\Bigl(S_{|m}^{(1)} \leq x \sqrt{n} - S_{|m}^{(2)}\Bigr) - \P\Bigl(Z_1 \overline{\sigma}_m \leq x - Z_2 {\overline{\varsigma}}_m \Bigr)\Bigr| \\&\leq \textbf{A} + \textbf{B} + \textbf{C},
\end{align*}
where $\textbf{A}, \textbf{B}, \textbf{C}$ are defined as
\begin{align*}
\textbf{A} = &\sup_{x \in \R}\Bigl|\E\Bigl[\P_{\FF_m}\Bigl(S_{|m}^{(1)}/\sqrt{n} \leq x  - S_{|m}^{(2)}/\sqrt{n}\Bigr) - \P_{\FF_m}\Bigl(Z_1 \sigma_{|m} \leq x  - S_{|m}^{(2)}/\sqrt{n} \Bigr)\Bigr] \Bigr|,\\
\textbf{B} = &\sup_{x \in \R}\Bigl|\E\Bigl[\P_{\FF_m}\Bigl(Z_1 \sigma_{|m} \leq  x  - S_{|m}^{(2)}/\sqrt{n}\Bigr) - \P_{\FF_m}\Bigl(Z_1 \overline{\sigma}_{m} \leq  x  - S_{|m}^{(2)}/\sqrt{n}\Bigr)\Bigr]\Bigr|,\\
\textbf{C} = & \sup_{x \in \R}\Bigl|\P\Bigl(S_{|m}^{(2)}/\sqrt{n} \leq  x  - Z_1 \overline{\sigma}_{m}\Bigr) - \P\Bigl(Z_2 {\overline{\varsigma}}_{m} \leq x - Z_1 \overline{\sigma}_{m} \Bigr)\Bigr|.
\end{align*}

We shall see that for all three parts $\textbf{A}, \textbf{B}, \textbf{C} \lesssim n^{-(p\wedge3)/2 + 1}$, which then clearly implies Theorem \ref{thm_m_dependent}.

\subsubsection{Part A}

The proof of part $\textbf{A}$ is divided into four major steps. 
\begin{proof}[Proof of $\textbf{A}$]
For $L > 0$, put $\mathcal{B}_{L} = \bigl\{{L}^{-1}\sum_{j = 1}^L \sigma_{j|m}^2 \geq \ss_m^2/8 \bigr\}$, and denote with $\mathcal{B}_{L}^{c}$ its complement. Since $S_{|m}^{(2)} \in \FF_m$, we obtain
\begin{align*}
\textbf{A} = &\sup_{x \in \R}\Bigl|\E\Bigl[\P_{\FF_m}\Bigl(S_{|m}^{(1)}/\sqrt{n} \leq x  - S_{|m}^{(2)}/\sqrt{n}\Bigr) - \P_{\FF_m}\Bigl(Z_1 \sigma_{|m} \leq x  - S_{|m}^{(2)}/\sqrt{n} \Bigr)\Bigr] \Bigr| \\&\leq \E\Bigl[\sup_{y \in \R}\Bigl|\P_{\FF_m}\Bigl(S_{|m}^{(1)}/\sqrt{n} \leq y\Bigr) - \P_{\FF_m}\Bigl(Z_1 \sigma_{|m} \leq y\Bigr)\Bigr|\ind(\mathcal{B}_{N})\Bigr] + 2\P\bigl(\mathcal{B}_{N}^{c}\bigr).
\end{align*}
An application of Lemma \ref{lem_sig_lower_bound} (improves and replaces Corollary 4.8 in \cite{jirak_be_aop_2016}, replacement is necessary due to weaker assumptions) yields $\P\bigl(\mathcal{B}_{N}^{c}\bigr) \lesssim e^{-C N}$ for some absolute constant $C > 0$. Since $N = n^{1 - \mathfrak{m}}$, $0 < \mathfrak{m} < 1$, by assumption, it thus suffices to treat
\begin{align}\label{defn_Delta_m}
\Delta_{|m} \stackrel{def}{=} \sup_{y \in \R}\Bigl|\P_{\FF_m}\Bigl(S_{|m}^{(1)}/\sqrt{n} \leq y\Bigr) - \P_{\FF_m}\Bigl(Z_1 \sigma_{|m} \leq y\Bigr)\Bigr|\ind(\mathcal{B}_{N}).
\end{align}

\textbf{\textit{Step 1:}} Berry-Esseen inequality. We use exactly (verbatim) the same argument as in \cite{jirak_be_aop_2016}. Since we also establish some necessary additional notation in doing so, we provide details. Let $\varphi_j(x) = \E_{\FF_m}e^{\ic x Y_j}$, and put $T = n^{p/2-1} c_T$, where $c_T > 0$ will be specified later.
Denote with $\Delta_{|m}^T$ the smoothed version of $\Delta_{|m}$ (cf. \cite{fellervolume2}) as in the classical approach. Since $\sigma_{|m}^2 \geq \ss_m^2/8 > 0$ on the set $\mathcal{B}_N$ by construction, the smoothing inequality (cf. \cite[Lemma 1, XVI.3]{fellervolume2}) is applicable, and it thus suffices to treat $\Delta_{|m}^T$.  Due to the independence of $(Y_j)_{1 \leq j \leq N}$ under $\P_{\FF_m}$, and since $\ind(\mathcal{B}_{N})\leq 1$, it follows that
\begin{align}\label{eq_berry_smoothing}
\bigl\|\Delta_{|m}^T\bigr\|_1 & \leq \int_{-T}^{T}\Bigl\|\prod_{j = 1}^N \varphi_j\bigl(\xi/\sqrt{N} \bigr) - \prod_{j = 1}^N e^{-\sigma_{j|m}^2 \xi^2/2N}\Bigr\|_1 /|\xi| d \xi.
\end{align}
Set $t = \xi/\sqrt{N}$. Standard computations and $|e^{\mathrm{i} x}| = 1$ then imply
\begin{align*}
&\Bigl\|\prod_{j = 1}^N \varphi_j(t) -  \prod_{j = 1}^N e^{-\sigma_{j|m}^2 t^2/2}\Bigr\|_1 \\&\leq \sum_{i = 1}^N\Bigl\|\prod_{j = 1}^{i-2}e^{-\sigma_{j|m}^2 t^2/2}\Bigr\|_1 \Bigl\|\varphi_i(t) -e^{-\sigma_{i|m}^2 t^2/2}\Bigr\|_1\Bigl\|\prod_{j = i+2}^N \bigl|\varphi_j(t)\bigr|\Bigr\|_1 \\&\leq N\Bigl\|\varphi_1(t) -e^{-\sigma_{1|m}^2 t^2/2}\Bigr\|_1\Bigl\|\prod_{j = N/2}^{N-1} \bigl|\varphi_j(t)\bigr|\Bigr\|_1 +
N \Bigl\|\prod_{j = 1}^{N/2-3}e^{-\sigma_{j|m}^2 t^2/2}\Bigr\|_1   \Bigl\|\varphi_1(t) -e^{-\sigma_{1|m}^2 t^2/2}\Bigr\|_1 \\&+ \Bigl\|\prod_{j = 1}^{N/2-3}e^{-\sigma_{j|m}^2 t^2/2}\Bigr\|_1 \Bigl\|\varphi_N(t) -e^{-\sigma_{N|m}^2 t^2/2}\Bigr\|_1 \stackrel{def}{=} I_N(\xi) + II_N(\xi) + III_N(\xi).
\end{align*}

We proceed by obtaining upper bounds for $I_N(\xi), II_N(\xi)$ and $III_N(\xi)$.\\
\\
\textbf{ \textit{Step 2:}}
In this step things are crucially different from \cite{jirak_be_aop_2016}. First, we use Lemma \ref{lem_sig_expansion} instead of Lemma 4.7 in \cite{jirak_be_aop_2016} for establishing the bound \eqref{eq_thm_aux_char_equation_bound_1} below. Note that the statements of the Lemmas are quite similar, yet the proofs are different in substantial parts. Secondly, one needs to employ the different argument used in the proof of Lemma \ref{lem_sig_expansion} to establish an analogous version of Lemma 4.10 in \cite{jirak_be_aop_2016}, subject to Assumption \ref{ass_dependence}. Since the proof of Lemma \ref{lem_sig_expansion} is given in full detail, we omit the details of showing the analogous version of Lemma 4.10 in \cite{jirak_be_aop_2016}. Once this has been done, we can proceed as in \cite{jirak_be_aop_2016} to derive
\begin{align}\label{eq_thm_aux_char_equation_bound_1}
\bigl\|\varphi_i(t) -e^{-\sigma_{i|m}^2 t^2/2}\bigr\|_1 \leq C |t|^p m^{-(p\wedge3)/2 + 1}, \quad i \in \{1,N\}.
\end{align}

\textbf{\textit{Step 3:}} The third step is devoted to bounding $\bigl\|\prod_{j = N/2}^{N-1} \bigl|\varphi_j(t)\bigr|\bigr\|_1$. This can be done by following \cite{jirak_be_aop_2016}, leading to the upper bound
\begin{align}\label{eq_thm_aux_char_equation_bound_6}
\bigl\|\prod_{j = N/2}^{N-1} \bigl|\varphi_j(t)\bigr|\bigr\|_1 \lesssim e^{-c_{\varphi,1}x^2 N/16} + e^{-\sqrt{N/32}\log 8/7}, \quad \text{for $x^2 < c_{\varphi,2}$.}
\end{align}
Here $x = t\sqrt{(m-l)/2m}$ for $c_{\varphi,3} \leq l \leq m$. It is important to emphasize that $c_{\varphi,1}, c_{\varphi,2}$ and $c_{\varphi,3}$ do not depend on $l,m$ and are strictly positive and finite. No adaptations are necessary here.
\\
\textbf{\textit{Step 4:}} Bounding and integrating $I_N(\xi),II_N(\xi),III_N(\xi)$:\\
Here, we largely follow again \cite{jirak_be_aop_2016}. However, we do need to employ our different results. For this reason, and since it is the final step and combines all previous results, we provide details. We first treat $I_N(\xi)$. Recall that $t = \xi/\sqrt{N}$, hence
\begin{align*}
|t|^{p\wedge3} m^{-(p\wedge3)/2 + 1} \lesssim |\xi|^{p\wedge3} n^{-(p\wedge3)/2 + 1}N^{-1}.
\end{align*}
By \eqref{eq_thm_aux_char_equation_bound_1} and \eqref{eq_thm_aux_char_equation_bound_6}, it then follows for $\xi^2 (m-l) < 2c_{\varphi,2}n$, that
\begin{align}
I_N(\xi) \lesssim |\xi|^{p\wedge3} n^{-(p\wedge3)/2 + 1}\Bigl(e^{-c_{\varphi,1}\xi^2 (m-l)/32m} + e^{-\sqrt{N/32}\log 8/7}\Bigr).
\end{align}
To make use of this bound, we need to appropriately select $l= l(\xi)$. Recall $N \thicksim n^{1 - \mathfrak{m}}$, $0 < \mathfrak{m} < 1$ by assumption. Let $0< \lambda < 1-\mathfrak{m}$, and put
\begin{align*}
l'(\xi) = \Big(\ind\bigl(\xi^2 < n^{\lambda} \bigr) + \Big(m - \frac{n}{\xi^2}\Big)\ind\bigl(\xi^2 \geq n^{\lambda} \bigr)\Big) \vee c_{\varphi,3}
\end{align*}
and $c_T^2 < 2 c_{\varphi,2}$. Then, with $l = l(\xi) = \lfloor l'(\xi) \rfloor$, we obtain from the above
\begin{align}
\int_{-T}^{T}\frac{I_N(\xi)}{\xi} d \xi \lesssim n^{-(p\wedge3)/2 + 1}.
\end{align}
In order to treat $II_N(\xi)$, let $N' = N/2 - 3$ and recall that
$$\mathcal{B}_{N'} = \bigl\{{N'}^{-1}\sum_{j = 1}^{N'} \sigma_{j|m}^2 \geq \ss_m^2/8 \bigr\}.$$
Denote with $\mathcal{B}_{N'}^{c}$ its complement. Then by Lemma \ref{lem_sig_lower_bound} and \eqref{eq_thm_aux_char_equation_bound_1}, it follows that
\begin{align}\label{eq_thm_aux_char_equation_bound_9}\nonumber
II_N(\xi) &\leq N\Bigl\|\varphi_1(t) -e^{-\sigma_{1|m}^2 \xi^2/2}\Bigr\|_1\Bigl\|\prod_{j = 1}^{N'}e^{-\sigma_{j|m}^2 \xi^2/2}\ind\bigl(\mathcal{B}_{N'}\bigr)\Bigr\|_1 \\&\nonumber+ N\Bigl\|\varphi_1(t) -e^{-\sigma_{1|m}^2 \xi^2/2}\Bigr\|_1 \P\bigl(\mathcal{B}_{N'}^c\bigr)\\& \lesssim |\xi|^{p\wedge3} n^{-(p\wedge3)/2 + 1} \Big(e^{-\ss_{m}^2 \xi^2/16} + e^{-C N}\Big), \quad C > 0.
\end{align}

From \eqref{eq_thm_aux_char_equation_bound_9} above, we thus obtain
\begin{align}\nonumber
\int_{-T}^{T}\frac{II_N(\xi)}{\xi} d\xi  &\lesssim  n^{-(p\wedge3)/2 + 1} \int_{|\xi| \leq T} |\xi|^{p\wedge3-1} \Bigl(e^{-\ss_{m}^2 \xi^2/16} + e^{-C N}\Bigr) d\xi \\& \lesssim n^{-(p\wedge3)/2 + 1} \Big(1 + N^{p\wedge3+1}e^{-C N}\Big) \lesssim n^{-(p\wedge3)/2 + 1},
\end{align}
since $N = n^{1 - \mathfrak{m}}$, $0 < \mathfrak{m} < 1$ by assumption. Similarly, one obtains the same bound for $III_N(\xi)$. This completes the proof of part $\textbf{A}$.
\end{proof}

\subsubsection{Part B}

\begin{proof}[Proof of $\textbf{B}$]
Using Lemma \ref{lem_sig_lower_bound}, we can repeat the argument of \cite{jirak_be_aop_2016}. Let
\begin{align*}
\Delta^{(2)}(x) \stackrel{def}{=} \E\Bigl[\P_{\FF_m}\Bigl(Z_1 \sigma_{|m} \leq  x  - S_{|m}^{(2)}/\sqrt{n}\Bigr) - \P_{\FF_m}\Bigl(Z_1 \overline{\sigma}_{m}\leq  x  - S_{|m}^{(2)}/\sqrt{n}\Bigr)\Bigr].
\end{align*}
Recall that
\begin{align*}
\P\bigl(\mathcal{B}_N^c\bigr) \lesssim e^{-C N} \lesssim n^{-(p\wedge3)/2 + 1}
\end{align*}
by Lemma \ref{lem_sig_lower_bound}. Using properties of the Gaussian distribution, we get
\begin{align*}
\textbf{B} &\leq \sup_{x\in\R}\bigl|\E \Delta^{(2)}(x) \ind(\mathcal{B}_{N}) \bigr| + \sup_{x\in\R}\bigl|\E \Delta^{(2)}(x) \ind(\mathcal{B}_{N}^c)\bigr|  \\&\lesssim  \E\bigl[\bigl|1/\sigma_{|m} - 1/\overline{\sigma}_{m}\bigr|\ind(\mathcal{B}_N)\bigr] + n^{-(p\wedge3)/2 + 1}.
\end{align*}
By $(a - b)(a+b) = a^2 - b^2$, H\"{o}lders inequality and Lemma \ref{lem_sig_expansion}, it follows that
\begin{align*}
\E\bigl[\bigl|1/\sigma_{|m} - 1/\overline{\sigma}_{m}\bigr|\ind(\mathcal{B}_{N})\bigr] \lesssim \bigl\|\sigma_{|m}^2  - \overline{\sigma}_{m}^2\bigr\|_{p/2} \lesssim n^{-(p\wedge3)/2 + 1}.
\end{align*}
Hence, all in all, $\textbf{B}\lesssim n^{-(p\wedge3)/2 + 1}$.
\end{proof}

\subsubsection{Part C}

\begin{proof}[Proof of $\textbf{C}$]
We can repeat the argument of \cite{jirak_be_aop_2016}. Due to the independence of $Z_1, Z_2$, we may rewrite $\textbf{C}$ as
\begin{align*}
\textbf{C} = \sup_{x \in \R}\Bigl|\Phi\Bigl(\bigl(x - S_{|m}^{(2)}/\sqrt{n}\bigr)/\overline{\sigma}_{m}\Bigr) -\Phi\Bigl(\bigl(x - Z_2 \overline{\varsigma}_{m}\bigr)/\overline{\sigma}_{m}\Bigr)\Bigr|,
\end{align*}
where $\Phi(\cdot)$ denotes the cdf of a standard normal distribution. This induces a 'natural' smoothing. The claim now follows by repeating the same arguments as in part $\textbf{A}$. Note however, that the present situation is much easier to handle, due to the already smoothed version (no Berry-Essen inequality is necessary, but can be used), and since $Y_k^{(2)}$, $k = 1,\ldots,N$ is a sequence of independent random variables. 
\end{proof}

\subsection{Proofs of Theorem \ref{thm_berry} and Proposition \ref{prop_unify}}\label{sec_proof_of_main_theorem}

The proof of Theorem \ref{thm_berry} mainly consists of constructing a good $m$-dependent approximation, and then verify the conditions of Theorem \ref{thm_m_dependent}. As in \cite{jirak_be_aop_2016}, the general idea is to use Lemma \ref{lem_siegi_decomposition} below. However, due to weaker assumptions, we will proceed differently compared to \cite{jirak_be_aop_2016} to control the various bounds and quantities.

Set $m = C n^{\mathfrak{m}}$ for some $C > 0$ and $0 < \mathfrak{m} < 1$. Moreover, put
\begin{align*}
\mathcal{E}_{kj} = \sigma(\epsilon_k, \ldots, \epsilon_{k-j+1}), \quad X_{kj} = \E\bigl[X_k \bigl|\mathcal{E}_{kj} \bigr], \quad S_{nj} = \sum_{k = 1}^n X_{kj},
\end{align*}
with the convention that $S_{n} = S_{n\infty}$. The sequence $(X_{km})_{k \in \Z}$ constitutes our $m$-dependent approximation. Further, let $s_{n\infty}^2 = n^{-1}\bigl\|S_n\bigr\|_2^2$ and recall $s_{nm}^2 = n^{-1}\bigl\|S_{nm}\bigr\|_2 = \overline{\sigma}_m^2 + \overline{\varsigma}_m^2$. We require the following auxiliary result (Lemma 5.1 in \cite{hormann_2009}).
\begin{lem}\label{lem_siegi_decomposition}
For every $\delta > 0$, every $m,n \geq 1$ and every $x \in \R$, the following estimate holds:
\begin{align*}
&\bigl|\P\bigl(S_n /\sqrt{n}\leq x s_{n\infty} \bigr) - \Phi(x)\bigr| \leq A_0(x,\delta) + A_1(m,n,\delta) \\&+ \max\bigl\{A_2(m,n,x,\delta) + A_3(m,n,\delta), A_4(m,n,x,\delta) + A_5(m,n,x,\delta)\bigr\},
\end{align*}
where
\begin{align*}
& A_0(x,\delta) = \bigl|\Phi(x) - \Phi(x + \delta)\bigr|;\\
& A_1(m,n,\delta) = \P\bigl(|S_n - S_{nm}| \geq \delta s_{n\infty} \sqrt{n}\bigr);\\
& A_2(m,n,x,\delta) = \bigl|\P\bigl(S_{nm} \leq (x + \delta)s_{n\infty}\sqrt{n}\bigr) - \Phi\bigl((x+\delta)s_{n\infty}/s_{nm}\bigr)\bigr|;\\
& A_3(m,n,x,\delta) = \bigl|\Phi\bigl((x+\delta)s_{n\infty}/s_{nm}\bigr) - \Phi(x+\delta)\bigr|;\\
& A_4(m,n,x,\delta) = A_2(m,n,x,-\delta) \quad \text{and} \quad A_5(m,n,x,\delta) =  A_3(m,n,x,-\delta).
\end{align*}
\end{lem}

\begin{proof}[Proof of Theorem \ref{thm_berry}]
As a preparatory result, note that by Lemma \ref{lem_sig_expressions_relations}
\begin{align}\label{eq_thm_gen_var_id}
n s_{n\infty}^2 = n \ss^2 + \OO\bigl(n^{1-\bd}\bigr).
\end{align}
By properties of the Gaussian distribution, we then obtain
\begin{align}\label{eq_thm_gen_var_id_2}
\sup_{x \in \R}\bigl|\Phi\bigl(x/\sqrt{\ss^2}\bigr) - \Phi\bigl(x/\sqrt{s_{n\infty}}\bigr)\bigr| \lesssim n^{-\bd}.
\end{align}
We first deal with $A_1(m,n,\delta)$. Proceeding exactly as in the proof of Theorem 1 in \cite{Wu_fuk_nagaev} (see equation 2.15), it follows that
\begin{align}\nonumber \label{eq_thm_gen_m_approx_S_n}
\bigl\|S_{n} - S_{nm}\bigr\|_p &\lesssim \sqrt{n} \sum_{k = m}^{\infty} \bigl\|X_k - X_k'\bigr\|_p \lesssim \sqrt{n} m^{-\bd} \sum_{k = m}^{\infty} k^{\bd}\bigl\|X_k - X_k'\bigr\|_p \\&\lesssim \sqrt{n} m^{-\bd}.
\end{align}
We thus conclude from Markov's inequality and \eqref{eq_thm_gen_var_id}, that
\begin{align*}
\P\bigl(|S_n - S_{nm}| \geq \delta s_{n\infty} \sqrt{n}\bigr) \lesssim  (\delta s_{n\infty} \sqrt{n} )^{-p} (\sqrt{n} m^{-\bd})^p \lesssim (\delta m^{\bd})^{-p},
\end{align*}
hence
\begin{align}\label{eq_thm_gen_a_1}
A_1(m,n,\delta) \lesssim (\delta n^{\bd \mathfrak{m}})^{-p}.
\end{align}
Next, we deal with $A_2(m,n,x,\delta)$. The aim is to apply Theorem \ref{thm_m_dependent} to obtain the result. In order to do so, we need to verify Assumption \ref{ass_dependence} for $X_{km}$.\\
\hyperref[A1]{\Aone}: Note first that $\E\bigl[X_{km}\bigr] = \E\bigl[X_k\bigr] = 0$. Moreover, Jensens inequality gives
\begin{align*}
\bigl\|X_{km}\bigr\|_p = \bigl\|\E\bigl[X_k\bigl| \mathcal{E}_{km}\bigr]\bigr\|_p \leq \bigl\|X_k\bigr\|_p < \infty.
\end{align*}
\hyperref[A2]{\Atwo}: Note that we may assume $k \leq m$, since otherwise $X_{km}^*- X_{km} = 0$. Put
\begin{align*}
\mathcal{E}_{km}^{*} = \sigma\bigl(\epsilon_j, \, 1 \leq j \leq k,\, \epsilon_j', \, k-m +1 \leq j \leq 0\bigr).
\end{align*}
Since $\E\bigl[X_k\bigl|\mathcal{E}_{km}\bigr]^* = \E\bigl[X_k^*\bigl|\mathcal{E}_{km}^*\bigr]$, it follows that
\begin{align}\nonumber
X_{km}^* - X_{km} &= \E\bigl[X_k^*\bigl|\mathcal{E}_{km}^*\bigr] - \E\bigl[X_k\bigl|\mathcal{E}_{km}\bigr] \\&= \E\bigl[X_{k}^* - X_{k}\bigl|\mathcal{E}_{km}^*\bigr] + \E\bigl[ X_{k}\bigl|\mathcal{E}_{km}^*\bigr] - \E\bigl[X_k\bigl| \mathcal{E}_{km} \bigr]\\&= \nonumber\E\bigl[X_{k}^* - X_{k}\bigl|\mathcal{E}_{km}^*\bigr] + \E\bigl[ X_{k}^*\bigl|\mathcal{E}_{km}\bigr] - \E\bigl[X_k\bigl|\mathcal{E}_{km}\bigr] \\& = \E\bigl[X_{k}^* - X_{k}\bigl|\mathcal{E}_{km}^*\bigr] + \E\bigl[X_k^* - X_k\bigl|\mathcal{E}_{km}\bigr].
\end{align}
Hence by Jensens inequality $\bigl\|X_{km}^* - X_{km}\bigr\|_p \leq 2 \bigl\|X_k - X_k^*\bigr\|_p$, which gives the claim by \hyperref[B2]{\Btwo}. Note that in exactly the same manner, one also obtains $\bigl\|X_{km}' - X_{km}\bigr\|_p \leq 2 \bigl\|X_k - X_k'\bigr\|_p$, which is relevant for our next argument.\\
\hyperref[A3]{\Athree}: Observe first that due to Lemma \ref{lem_sig_expressions_relations} we have $\bigl|\sum_{k = 1}^{\infty}\E\big[X_{0m} X_{km}\bigr]\bigr| < \infty$, uniformly in $m$. Since $\|X_k - X_{km}\|_p \to 0$, Cauchy-Schwarz yields
\begin{align*}
\bigl|\E\big[X_{0m} X_{km}\bigr] - \E\big[X_{0} X_{k}\bigr] \bigr| \to 0
\end{align*}
for any fixed $k \in \N$ and large enough $m$. We thus conclude from Lemma \ref{lem_sig_expressions_relations} and arguments from its proof
\begin{align*}
\sum_{k \in \Z} \bigl|\E[X_{0} X_{k}] - \E[X_{0m} X_{km}]\bigr| = \oo\bigl(1\bigr) + \sum_{|k| > m} \bigl|\E[X_{0} X_{k}] \bigr| = \oo\bigl(1\bigr).
\end{align*}
Hence \hyperref[A3]{\Athree} follows from \hyperref[B3]{\Bthree}. Since $m \thicksim n^{\mathfrak{m}}$ with $0 < \mathfrak{m} < 1$, we may thus apply Theorem \ref{thm_m_dependent} which yields
\begin{align}\label{eq_thm_gen_a_2}
\sup_{x \in \R} A_2(m,n,x,\delta) \lesssim n^{-p/2 + 1}.
\end{align}
Next, we deal with $A_3(m,n,x,\delta)$. Properties of the Gaussian distribution function give
\begin{align*}
\sup_{x \in \R}A_3(m,n,x,\delta) \lesssim \delta + \bigl|s_{n\infty}^2 - s_{nm}^2\bigr|.
\end{align*}
However, by the Cauchy-Schwarz inequality, \eqref{eq_thm_gen_m_approx_S_n} and Lemma \ref{lem_wu_original}, it follows that
\begin{align}
\bigl|s_{n\infty}^2 - s_{nm}^2\bigr| \leq n^{-1}\bigl\|S_n - S_{nm}\bigr\|_2 \bigl\|S_n + S_{nm}\bigr\|_2 \lesssim n^{1/2 -\bd \mathfrak{m} -1/2}\lesssim n^{-\bd \mathfrak{m}},
\end{align}
and we thus conclude
\begin{align}\label{eq_thm_gen_a_3}
\sup_{x \in \R}A_3(m,n,x,\delta) \lesssim \delta + n^{-\bd \mathfrak{m}}.
\end{align}
Finally, setting $\delta = n^{-1/2}$, standard arguments involving the Gaussian distribution function yield
\begin{align}\label{eq_thm_gen_a_0}
\sup_{x \in \R}A_0(x,\delta) \lesssim \delta = n^{-1/2}.
\end{align}
Piecing together \eqref{eq_thm_gen_a_1},\eqref{eq_thm_gen_a_2},\eqref{eq_thm_gen_a_3} and \eqref{eq_thm_gen_a_0}, Lemma \ref{lem_siegi_decomposition} and \eqref{eq_thm_gen_var_id_2} yield
\begin{align}\label{eq_thm_gen_a_10}
\sup_{x \in \R}\bigl|\P\bigl(S_n /\sqrt{n}\leq x \bigr) - \Phi\bigl(x/\ss\bigr) \bigr|\lesssim n^{-p/2 + 1} + n^{-(\bd \mathfrak{m} -1/2)p}.
\end{align}
Selecting $\mathfrak{m}$ sufficiently close to one, this completes the proof since $(\bd -1/2)p > (p \wedge 3)/2 - 1$ by assumption.
\end{proof}

\begin{proof}[Proof of Proposition \ref{prop_unify}]
For any $k \geq 1$, we have
\begin{align*}
k^{\bd}(2k - k) \vartheta_{2k}'(p) \leq \sum_{l = k}^{2k} l^{\bd} \sup_{j \geq l}\vartheta_j'(p)
\end{align*}
by monotonicity. Hence we obtain
\begin{align*}
\limsup_{k \to \infty} k^{1 + \bd} \vartheta_k'(p) = 0.
\end{align*}
From Theorem 1 in \cite{wu_2005}, we have the (adjusted) inequality
\begin{align}\label{eq:prop:unify:1}
\sup_{k \in \Z}\bigl\|X_k - X_k^{(l,*)}\bigr\|_p^2 \lesssim \sum_{j \geq l} \sup_{k \in \Z}\bigl\|X_k - X_k^{(j,')}\bigr\|_p^2,
\end{align}
hence, since $\bd > 1/2 + (p\wedge3)/2p - 1/p > 1/2$, we deduce from the above that for sufficiently small $\ad > 0$
\begin{align*}
\sum_{k = 1}^{\infty} k^{\ad} \vartheta_k^{\ast}(p) \lesssim \sum_{k = 1}^{\infty} k^{\ad - \bd - 1/2} < \infty,
\end{align*}
and hence the first claim follows. The second claim is an obvious consequence of \eqref{eq:prop:unify:1}.
\end{proof}

\subsection{Proof of Theorem \ref{thm_lower_bound}}\label{sec:proof:lower:bound}

Throughout this section, formally denote with $X_k = \sum_{j = 0}^{\infty} \alpha_j \epsilon_{k-j}$, where $(\epsilon_i)_{i \in \Z}$ is an i.i.d. sequence with $\E \epsilon_i = 0$, $G_k = \sum_{j = 0}^{\infty} \alpha_j \xi_{k-j}$, where $(\xi_i)_{i \in \Z}$ is a sequence of standard i.i.d. Gaussian random variables. Theorem \ref{thm_lower_bound} is an (almost) immediate consequence of the following result.

\begin{prop}\label{prop_linear_lower}
Suppose that $\sum_{j \geq 0} |\alpha_j| < \infty$, $\sum_{j \geq 0} \alpha_j \neq 0$, and $\E \epsilon_i^2 = 1$, $\E |\epsilon_j|^3 < \infty$. Then
\begin{align*}
\sup_{x \in \R}\Big|\P\Bigl(\sum_{k = 1}^n X_k \leq x \sqrt{n \ss^2}\Bigr) - \Phi\bigl(x\bigr)\Bigr| \lesssim {n^{-1/2}} + n^{-1}\sum_{j \geq 1} (j \wedge n) |\alpha_j|,
\end{align*}
where $\ss = \big|\sum_{j \geq 0} \alpha_j\big|$. Moreover, the above bound is sharp, that is, there exists a linear process with the above properties, where the bound is reached (up to a multiplicative constant).
\end{prop}

\begin{proof}[Proof of Theorem \ref{thm_lower_bound}]
Let $1 + \bd < a < 3/2$. Then setting $\alpha_j = j^{-a}$, we have
\begin{align*}
\sum_{j \geq 1} (j \wedge n) |\alpha_j| \thicksim n^{-a + 2} \quad \text{and} \quad \sum_{k = 1}^{\infty} k^{\bd}\big\|X_k - X_k'\big\|_p < \infty.
\end{align*}
Since $a< 3/2$, the claim follows from Proposition \ref{prop_linear_lower} (with $\delta = 3/2 - a$).
\end{proof}

It remains to establish Proposition \ref{prop_linear_lower}. To this end, we require the following Lemma.

\begin{lem}\label{lem_BE_for_linear}
Grant the assumptions of Proposition \ref{prop_linear_lower}. Then
\begin{align*}
\Delta_n' = \sup_{x \in \R} \Bigl|\P\Bigl(n^{-1/2}\sum_{k = 1}^n X_k \leq x \Bigr) - \P\Bigl(n^{-1/2}\sum_{k = 1}^n G_k \leq x \Bigr) \Bigr| \lesssim n^{-1/2}.
\end{align*}
\end{lem}

\begin{rem}
Following the proof of Lemma \ref{lem_BE_for_linear}, one may derive the bound $n^{-1/2}$ also subject to different assumptions. In particular, neither $\sum_{j \geq 0}|\alpha_j| < \infty$ nor $\sum_{j \geq 0} \alpha_j \neq 0$ are necessary conditions.
\end{rem}

\begin{proof}[Proof of Lemma \ref{lem_BE_for_linear}]
Let $A = \sum_{j \geq 0} \alpha_j$. With the convention that $\alpha_j = 0$ for $j < 0$, we have the decomposition
\begin{align*}
\sum_{k = 1}^n X_k &= \sum_{i = 1}^n \sum_{j = 1-i}^{n-i} \alpha_j \epsilon_i + \sum_{i = 0}^{\infty} \sum_{j = 1+i}^{n + i}\alpha_j \epsilon_{-i}\\&\stackrel{def}{=} \sum_{i = 1}^n {A}_{n,i}\epsilon_i + \sum_{i = -\infty}^{0} {A}_{n,-i}\epsilon_{i}.
\end{align*}
Since $|{A}_{n,i}| \leq \sum_{j \geq 0}|\alpha_j| < \infty$, we have $\E |{A}_{n,i}\epsilon_i|^3 < \infty$ uniformly in $n,i$, and
\begin{align*}
\E \sum_{i = - \infty}^{0}|A_{n,-i} \epsilon_i|^3 \lesssim \sum_{j \geq 0} n |\alpha_j| \lesssim n.
\end{align*}
In addition, straightforward computations imply
\begin{align*}
\big\|\sum_{i = 1}^n {A}_{n,i}\epsilon_i\big\|_2^2 = \big(n - o(n)\big) A^2.
\end{align*}
By the above and the existing Berry-Esseen literature (cf. \cite{petrov_book_1995}, Theorem 5.4), we thus obtain
\begin{align*}
\Delta_n' \lesssim \frac{1}{(n A^2)^{3/2}} \Big(\sum_{i = 1}^n \E \big|{A}_{n,i}\epsilon_i\big|^3 + n \Big) \lesssim n^{-1/2}.
\end{align*}
\end{proof}

\begin{proof}[Proof of Proposition \ref{prop_linear_lower}]
Recall $\ss = \big|\sum_{j \geq 0} \alpha_j\big|$. Due to Lemma \ref{lem_BE_for_linear}, it suffices to establish
\begin{align}\label{eq:prop_linear_lower:1}
\sup_{x \in \R} \Bigl|\P\Bigl(n^{-1/2}\sum_{k = 1}^n G_k \leq x \Bigr) - \Phi\Big(\frac{x}{\ss}\Big) \Bigr| \lesssim n^{-1}\sum_{j \geq 1} (j \wedge n) |\alpha_j|
\end{align}
for the upper bound. The (first) claim then follows from the triangle inequality. Let $s_n^2 = n^{-1}\bigl\|\sum_{k = 1}^n G_k\bigr\|_2^2$. From the proof of Lemma \ref{lem_sig_expressions_relations}, we have
\begin{align*}
\bigr|\ss^2 - s_n^2\bigl| \thicksim  n^{-1}\sum_{j \geq 1} (j \wedge n) |\alpha_j|.
\end{align*}
\eqref{eq:prop_linear_lower:1} now follows by employing a Taylor expansion. If $n^{-1/2}\sum_{j \geq 1} (j \wedge n) |\alpha_j| \to \infty$, another Taylor expansion also shows that the bound in Proposition \ref{prop_linear_lower} is sharp. If $n^{-1/2}\sum_{j \geq 1} (j \wedge n) |\alpha_j| = \OO(1)$, one may resort to the standard i.i.d. binomial example, exhibiting the rate $n^{-1/2}$.
\end{proof}

\section*{Acknowledgements}

I would like to thank Christophe Cuny, Kasun Fernando and Florence Merlev\`{e}de for constructive comments and pointing out references. Special thanks to the reviewers for their very helpful remarks and
suggestions, significantly improving the quality of this note.

\end{document}